\documentclass{article}

\usepackage{arxiv}

\usepackage[utf8]{inputenc} 
\usepackage[T1]{fontenc}    
\usepackage{hyperref}       
\usepackage{url}            
\usepackage{booktabs}       
\usepackage{amsfonts}       
\usepackage{nicefrac}       
\usepackage{microtype}      

\usepackage{amsthm}
\usepackage{amsmath}
\usepackage{graphicx}
\usepackage{caption}
\usepackage{subcaption}
\usepackage{amssymb}
\usepackage{mdframed}
\usepackage{tikz}
\usetikzlibrary{positioning, arrows.meta, angles, quotes}

\usepackage{setspace}
\DeclareMathOperator*{\argmax}{argmax}
\DeclareMathOperator*{\argmin}{argmin}

\DeclareMathOperator{\rank}{rank}

\theoremstyle{plain}
\newtheorem{thm}{Theorem}[section]
\newtheorem{defn}[thm]{Definition}

\newtheorem{lemma}[thm]{Lemma}
\newtheorem{coro}[thm]{Corollary}
\newtheorem{prop}[thm]{Proposition}

\title{Optimality of Right-Invariant Priors}

\author{
 Jannis Bolik \\
 Department of Statistical Science, Duke University\\
  Department of Computer Science, ETH Zurich\\
   \And
 Thomas Hofmann \\
  Department of Computer Science, ETH Zurich\\
}

\begin{document}
\maketitle
\begin{abstract}
We discuss optimal prediction for families of probability distributions with a locally compact topological group structure. Right-invariant priors were previously shown to yield a posterior predictive distribution minimizing the worst-case Kullback-Leibler risk among all predictive procedures. However, the assumptions for the proof are so strong that they rarely hold in practice and it is unclear when the density functions used in the proof exist. Therefore, we provide a measure-theoretic proof, establishing adequate regularity assumptions. As applications, we show a strong optimality result for next-sample prediction for multivariate normal distributions and Gaussian Process regression with fixed lengthscale. We also discuss uniqueness and numerically evaluate prediction with right-invariant priors against other objective priors and plug-in prediction.
\end{abstract}

\section{Introduction}
The key idea in Bayesian prediction is to use Bayes' rule to obtain a probability distribution over a measurable family of probability distributions, and then again to derive a probability distribution over a random variable we want to predict. Having observed $n$ samples in $\mathcal{Y}$ and predicting the next $m$ samples the posterior and posterior predictive distributions are given by: 
$$p(\theta \mid \mathbf{y}_{1:n}) \propto p(\mathbf{y}_{1:n} \mid \theta) p(\theta) \qquad p(\mathbf{y}^*_{1:m} \mid \mathbf{y}_{1:n}) = \int p(\mathbf{y}_{1:m}^* \mid \mathbf{y}_{1:n}, \theta) p(\theta \mid \mathbf{y}_{1:n}) d\mu(\theta)$$

We now want to highlight a different perspective on the Bayesian posterior predictive from an optimization perspective. \cite{aitchison1975goodness} In this parametric setting, we are assuming that there exists a parameter $\theta^*$ such that $p(\mathbf{y}_{1:n} \mid  \theta^*)$ describes the data-generating distribution. Therefore, the optimal predictive distribution would simply be given by $p(\mathbf{y}_{1:n} \mid  \theta^*)$. While we do not know $\theta^*$, we can evaluate any predictive distribution $q(\mathbf{y}_{1:m}^* ; \mathbf{y}_{1:n})$ against $p(\mathbf{y}_{1:n} \mid  \theta)$ conditional on some $\theta$ being the data-generating parameters $\theta^*$:
\begin{equation}
\label{eq:optimality}
\mathbb{E}_{(\mathbf{y}_{1:n}, \mathbf{y}_{1:m}^*) \sim p(\cdot \mid \theta^*)}\left[\log q(\mathbf{y}_{1:m}^* ; \mathbf{y}_{1:n}) - \log p(\mathbf{y}_{1:m}^* \mid \mathbf{y}_{1:n}, \theta^*)\right]
\end{equation}
Here $q(\mathbf{y}_{1:m}^* ; \mathbf{y}_{1:n})$ denotes a function from $\mathbf{y}_{1:n}$ to a probability measure on $\mathbf{y}_{1:m}^*$. This can also be seen as a Kullback-Leibler divergence between $q$ and $p$ in expectation over $p(\mathbf{y}_{1:n} \mid \theta)$. Of course, the best choice of $q(\mathbf{y}_{1:m}^* ; \mathbf{y}_{1:n})$ depends on the parameters $\theta^*$ that generate the data. 

In particular, if we do not know $\theta^*$, we could assume that it is drawn from a distribution, the \emph{prior distribution} $\theta^* \sim p(\theta^*)$. This exactly matches the idea in Bayesian statistics that the prior distribution reflects our prior beliefs about the data-generating distribution before having observed any data. Then, the Bayesian posterior predictive distribution was shown to be optimal in terms of (\ref{eq:optimality}) in expectation: \cite{aitchison1975goodness}
\begin{equation}
\label{eq:expectation}
\argmax_q \mathbb{E}_{\theta^* \sim p(\theta)}\left[\mathbb{E}_{(\mathbf{y}_{1:n}, \mathbf{y}_{1:m}^*) \sim p(\cdot \mid \theta^*)}\left[\log q(\mathbf{y}_{1:m}^* ; \mathbf{y}_{1:n}) - \log p(\mathbf{y}_{1:m}^* \mid \mathbf{y}_{1:n}, \theta^*)\right]\right] = p(\mathbf{y}_{1:m}^* \mid \mathbf{y}_{1:n})
\end{equation}
However, it is often difficult to choose a prior distribution $p(\theta)$ in a meaningful way over anything more complex than a location variable, since it needs to be defined with respect to a usually arbitrarily chosen measure $\mu(\theta)$. In particular, to accurately express subjective prior beliefs, we need to know which $p(\theta)$ would express a neutral prior belief. 

Instead of maximizing (\ref{eq:optimality}) in expectation over $p(\theta)$, we will therefore consider the problem of maximizing the predictive distribution's worst-case performance:
\begin{equation}
\label{eq:maximin}
\argmax_q \min_{\theta^*} \mathbb{E}_{(\mathbf{y}_{1:n}, \mathbf{y}_{1:m}^*) \sim p(\cdot \mid \theta^*)}\left[\log q(\mathbf{y}_{1:m}^* ; \mathbf{y}_{1:n}) - \log p(\mathbf{y}_{1:m}^* \mid \mathbf{y}_{1:n}, \theta^*)\right]
\end{equation}
Changing the sign, we call
$$-\mathbb{E}_{(\mathbf{y}_{1:n}, \mathbf{y}_{1:m}^*) \sim p(\cdot \mid \theta^*)}\left[\log q(\mathbf{y}_{1:m}^* ; \mathbf{y}_{1:n}) - \log p(\mathbf{y}_{1:m}^* \mid \mathbf{y}_{1:n}, \theta^*)\right]$$
\begin{equation}
\label{eq:risk}
=\mathbb{E}_{\mathbf{y}_{1:n} \sim p(\cdot \mid \theta^*)}\left[ D_{KL}(p(\mathbf{y}^*_{1:m} \mid \mathbf{y}_{1:n}, \theta^*) \parallel q(\mathbf{y}^*_{1:m} \mid \mathbf{y}_{1:n}))\right]    
\end{equation}
the \emph{predictive risk}, also known as Kullback-Leibler risk.

\subsection{Prior Work}
Unlike for (\ref{eq:expectation}), it is generally difficult to find a predictive procedure $q(\mathbf{y}_{1:m}^* ; \mathbf{y}_{1:n})$ that solves (\ref{eq:maximin}). However, it turns out to be feasible if the likelihood has a locally compact topological group structure. 

If $y_1, ..., y_n, y^*_1, ..., y^`_m$ are i.i.d., this means that there exists a locally compact topological group $G$ acting on the standard Borel space $\mathcal{Y} = (Y, \mathfrak{T}, \Sigma)$, and $\Theta$ such that:
$$\forall g \in G, \theta \in \Theta, A \in B(\mathcal{Y}): p(A \mid \theta) = p(gA \mid g\theta)$$
In particular, this means that the group actions $\alpha_\mathcal{Y}: G \times \mathcal{Y} \rightarrow \mathcal{Y}$ and $\alpha_{\Theta}: G \times \Theta \rightarrow \Theta$ have to be measurable functions (in $\mathcal{Y}$ and $\Theta$ respectively). The group action on the product spaces $\mathcal{Y}^n$ and $\mathcal{Y}^m$ is defined element-wise. 

Furthermore, we define \emph{invariant predictive procedures} as predictive procedures $q$ that fulfill:
$$\forall g \in G, \mathbf{y} \in \mathcal{Y}^n, A \in B(\mathcal{Y}^m): q(A; \mathbf{y}_{1:n}) = q(gA; g\mathbf{y}_{1:n})$$

Komaki showed the first result in this context. \cite{komaki2002bayesian} His proof requires several additional assumptions on the group actions for next-sample prediction ($m=1$):
\begin{enumerate}
    \item $\alpha_\mathcal{Y}$ is free on $\mathcal{Y}$
    \item  $\alpha_\Theta$ is simply transitive
\end{enumerate}
A free group action is an action on a space $X$ such that $\forall g \in G, x \in X: gx=x \Rightarrow g=e$ where $e$ is the group's identity element. A simply transitive group action is a group action that for $x_1, x_2 \in X$ there exists a unique $g \in G$ such that $g x_1 = x_2$. Under these assumptions and several implicit assumptions on the existence of density functions, he shows that the Bayesian posterior predictive with a right-invariant measure on $\Theta$ as an improper prior optimizes (\ref{eq:maximin}) among the \emph{invariant} predictive distributions. 

Later, he shows a similar result under the weaker assumption that $\alpha_\Theta$ is only assumed to be free instead of simply transitive. \cite{komaki2004noninformative} However, this no longer optimizes (\ref{eq:maximin}) but is a hybrid of (\ref{eq:maximin}) and (\ref{eq:expectation}), where a prior over the orbits of $\Theta$ under $\alpha_\Theta$ is defined.

Komaki's result gained additional importance through a result by Liang \& Barron. \cite{liang2004exact} They do not assume $\alpha_\mathcal{Y}$ to be free or $\alpha_\Theta$ to be transitive but they have a different set of additional assumptions:
\begin{enumerate}
    \item $G$ is amenable
    \item for all $\theta \in \Theta$, $p(y \mid \theta)$ is continuous in $y$ and there exists a $\theta \in \Theta$ for which it is strictly positive
\end{enumerate}
An amenable group is a group that admits a left-invariant averaging operation. Under these assumptions, they show that the best \emph{invariant} predictive procedure, as determined by Komaki's result,  is also the best procedure overall in terms of (\ref{eq:maximin}). They have a slightly different set of assumptions on $\alpha_\mathcal{Y}$ but we show that they are always fulfilled in our setting in appendix \ref{sec:liang-assumptions}.

It should be noted that there often exist non-invariant Bayesian posterior predictive distributions with a different prior that also optimize (\ref{eq:maximin}) and at the same time perform strictly better (risk is always smaller or equal and sometimes smaller). \cite{komaki2001shrinkage} These are known as shrinkage priors. Sometimes, they can even be proper. \cite{liang2002exact} There are also some known rules on how such shrinkage priors can be constructed in general. \cite{komaki2006shrinkage} However, further improvements beyond optimal worst-case performance, as achieved by shrinkage priors, are not the focus of this paper.

\subsection{Adequacy of Assumptions} \label{sec:adequacy}
The proofs of optimality theorems make several technical assumptions. We now want to determine whether they are fulfilled in practice. The canonical example of a likelihood with an amenable locally compact topological group structure is the multivariate normal distribution. Komaki in fact claims that his theorem is a generalization of a previous proposition on optimal priors for the $d$-dimensional normal distribution by Geisser. \cite{geisser1993predictive} Geisser claims that the optimal prior on $(\mu, \Sigma)$ such that the posterior predictive is invariant under the affine group $\mathrm{Aff}(d, \mathbb{R})$ is given by $\det(\Sigma)^{-(d+1)/2}$. It is well known that the affine transformation $A\mathbf{x}+\mathbf{b}$ of a multivariate normal random variable $\mathcal{N}(\mu, \Sigma)$ is given by $\mathcal{N}(A\mu+\mathbf{b}, A \Sigma A^T)$. However, this group action on the parameter space is not simply transitive. Therefore, we use a different group to provide the multivariate normal distribution's group structure:

\begin{defn}
The group $G_\mathcal{N}$ consists of upper triangular matrices with positive diagonal $V \in \mathbb{R}^{n \times n}$ and $\mathbf{m} \in \mathbb{R}^n$. The group operation is defined as $(V_2,\mathbf{m_2}) \cdot (V_1, \mathbf{m_1}) = (V_2 V_1, V_2 \mathbf{m_1} + \mathbf{v_2})$.
\end{defn}

\begin{prop}
The group $G_\mathcal{N}$ provides the multivariate normal likelihood with a locally compact topological group structure with the following group actions: Consider the parametrization of the multivariate normal distribution in terms of the upper triangular Cholesky decomposition of its covariance matrix $U \in \mathbb{R}^{n \times n}$ (has positive diagonal) and its mean $\boldsymbol\mu \in \mathbb{R}^n$. $\alpha_\Theta$ is defined analogously to the group action: $(V, \mathbf{m}) \cdot (U, \boldsymbol\mu) = (VU, V\boldsymbol \mu+\mathbf{m})$. For the observation space $\mathcal{Y}=\mathbb{R}^n$ we define $\alpha_\mathcal{Y}$ as $(V, \mathbf{m}) \cdot \mathbf{y} = V \mathbf{y} + \mathbf{m}$. 
\end{prop}
\begin{proof}
Since $U$ and $V$ are invertible, $\alpha_\Theta$ is simply transitive. $G_\mathcal{N}$'s actions are all measurable. For the group invariance property, we will consider the probability density function, which means we also need to account for the change of variables. In this case, the Jacobian is $V$. We now have:
$$(2\pi)^{-k/2}\det (VU)^{-1} \, \exp \biggl( -\frac{1}{2} ((V\mathbf{x}+\mathbf{m}) - (V\boldsymbol\mu+\mathbf{m}))^\mathrm{T} \left(VU (VU)^T\right)^{-1}$$
$$ ((V\mathbf{x}+\mathbf{m}) - (V\boldsymbol\mu+\mathbf{m})) \biggr)\det(V)$$
$$=(2\pi)^{-k/2}\det (U)^{-1} \, \exp \left( -\frac{1}{2} (\mathbf{x} - \boldsymbol\mu)^\mathrm{T} \left(U U^T\right)^{-1}(\mathbf{x} - \boldsymbol\mu) \right)$$
This proves the invariance. 
\end{proof}
\begin{prop}
$G_\mathcal{N}$ is an amenable Lie group.
\end{prop}
\begin{proof}
From the rules of block matrix multiplication, we see that the group can be seen as a matrix group by considering the following:
$$\begin{pmatrix}
    V_2 & \mathbf{m_2} \\
    \mathbf{0}^T & 1
\end{pmatrix}
\begin{pmatrix}
    V_1 & \mathbf{m_1} \\
    \mathbf{0}^T & 1
\end{pmatrix}
=\begin{pmatrix}
    V_2 V_1 & V_2 \mathbf{m_1}+\mathbf{m_2} \\
    \mathbf{0}^T & 1
\end{pmatrix}$$
Therefore the matrix can be seen as a closed subgroup of the Lie group of real upper triangular matrices which is known to be a solvable Lie group. Since subgroups of solvable Lie groups are solvable Lie groups and solvable groups are amenable, $G$ is an amenable Lie group.
\end{proof}

Let us now check whether the individual assumptions of the theorems are fulfilled.

\begin{lemma}
(\cite{nabeya1986transformations}, Theorem 2.2) Let $f: \mathbb{R}^n \rightarrow \mathbb{R}^n$ be bijective and bimeasurable and for all positive definite $\Sigma$ and two distinct means $\mu \in \{\mu_1, \mu_2\}$ we have that $f(X), X \sim \mathcal{N}(\mu, \Sigma)$ is normally distributed. Then $f$ is an affine map $f(x) = A\mathbf{x}+\mathbf{b}$ where $A\in GL(n, \mathbb{R})$ and $b  \in \mathbb{R}^n$. 
\end{lemma}
\begin{prop} \label{prop:unique-group}
For the multivariate normal likelihood ($d>1$), if for some group $G$ and its group actions $\alpha_\mathcal{Y}$ and $\alpha_\Theta$ it holds that 
\begin{enumerate}
    \item $\alpha_\mathcal{Y}$ measurable on $\mathcal{Y}$
    \item $\alpha_\Theta$ is simply transitive and measurable
    \item  $\forall g \in G, \theta \in \Theta, A \in \sigma(\mathcal{Y}): p(A \mid \theta) = p(gA \mid g\theta)$.
\end{enumerate}
Then, $G=G_\mathcal{N}$, $\alpha_Y$ acts affinely on $\mathbb{R}^d$ and $\alpha_Y$ cannot be free.
\end{prop}
\begin{proof}
As any group action $\alpha_\mathcal{Y}$ has to be bijective in $\mathcal{Y}$ since the group action for the inverse element has to be the inverse function (which then also has to be measurable). Since $\alpha_\mathcal{Y}$ has to preserve normality for all multivariate normal distributions, we know from the lemma that it has to be an affine map $A\mathbf{x}+\mathbf{b}$ where $A\in GL(n, \mathbb{R})$ and $b  \in \mathbb{R}^n$. Consider the stabilizing subgroup of the standard multivariate normal distribution $\mathcal{N}(\mathbf{0}, I)$. The stabilizing subgroup is the orthogonal group $O(n)$. We can decompose any linear transformation using the $RQ$-decomposition (a variant of $QR$-decomposition). If we cancel out the orthogonal component first $AQ^{-1} = RQQ^{-1} =R$  we get exactly the set of triangular matrices with positive diagonal entries. Therefore, the group we have given is the unique choice (up to isomorphism) for the group structure of the multivariate normal likelihood. We have also shown that $\alpha_\mathcal{Y}$ has to be affine. However, $\alpha_\mathcal{Y}$ cannot free since for example $\alpha_\mathcal{Y}(\mathbf{x}) = A\mathbf{x}+\mathbf{b}$ leaves $\mathbf{x} = \mathbf{0}$ unchanged for any $A$ as long as $\mathbf{b}=0$ (which is necessary for multivariate normal distributions with the same mean).
\end{proof}
This shows that Komaki's assumptions have to be refined if we want to apply the theorem to the multivariate normal distribution. Additionally, Komaki implicitly assumes the existence of various density functions throughout his proof.

For Liang's theorem, we have already shown that $G$ is amenable, and it is also clear that the probability density function of the multivariate normal distribution with respect to the Lebesgue measure is continuous and positive.

\section{Main Result}
The following theorem is a variant of Komaki's theorem \cite{komaki2002bayesian} but with more precise assumptions. In particular, our assumptions will guarantee the existence of density functions (Radon-Nikodym derivatives) whose existence Komaki implicitly assumes. Additionally, we refine Komaki's freeness assumptions to allow application to next-sample prediction for the multivariate normal distribution. We also prove an additional property that holds for \emph{any} invariant predictive procedure and assume a more general group structure that does not require $p(\mathbf{y}^*_{1:m}, \mathbf{y}_{1:n} \mid \theta) = p(\mathbf{y}^*_{1:m} \mid \theta)p(\mathbf{y}_{1:n} \mid \theta)$.

\subsection{Mathematical Background}
An important theorem in this context is the Radon-Nikodym theorem, which also allows us to define KL-divergence in a measure-theoretic sense:
\begin{defn}
(\cite{fremlin2000measure} 232B) Let $(\mu, X, \Sigma)$ be a  measure space and $\nu: \Sigma \rightarrow \mathbb{R}$ be a countably additive function. Then we define the absolute continuity of $\nu$ with respect to $\mu$ as
$$\nu \ll \mu := \forall A \in \Sigma: \mu(A)=0 \Rightarrow \nu(A)=0$$
\end{defn}
\begin{thm}
(\cite{fremlin2000measure} 232 F, Hd) Radon-Nikodym Theorem: Let $(X, \Sigma, \mu)$ be a $\sigma$-finite measure space and $\nu: \Sigma \rightarrow \mathbb{R}$ a function. Then there is a $\mu$-integrable function $f$ such that $\nu E=\int_E f$ for every $E \in \Sigma$ iff $\nu$ is countably additive and absolutely continuous with respect to $\mu$. We call such an $f$ a Radon-Nikodym derivative $\frac{d\nu}{d\mu}$. If there are multiple Radon-Nikodym derivatives $\frac{d\nu}{d\mu}$, then there they are equal $\mu$-almost everywhere.
\end{thm}
\begin{defn}
Let $\mu$ and $\nu$ be two measures on a space $(X, \Sigma)$ where $\mu$ is $\sigma$-finite. If $\nu \ll \mu$, then
$$D_{KL}(\mu \parallel \nu) := -\int \log\left(\frac{d\nu}{d\mu}\right) d\mu$$
where $\frac{d\nu}{d\mu}$ is a Radon-Nikodym derivative. Otherwise, $D_{KL}(\mu \parallel \nu) := \infty$.
\end{defn}

To rigorously decompose the observation space into its orbits we also need the disintegration theorem:
\begin{thm}
(\cite{fremlin2000measure} 452 E, F, G, P) Disintegration: Let $(X, \mathfrak{T}, \Sigma, \mu)$ be a standard Borel probability space, $(Y, \mathfrak{S}, \mathrm{T}, \nu)$ an analytic Borel probability space and $f: X \rightarrow Y$ an inverse-measure-preserving function. Then there exists a disintegration $\left\langle\mu_y\right\rangle_{y \in Y}$ of $\mu$ over $\nu$ such that 
\begin{itemize}
    \item every $\mu_y$ is a Radon measure on $X$
    \item  for $\nu$-almost all $y \in Y$, $\mu_y$ is a conegligible probability measure on $f^{-1}(y)$ 
    \item for every $[-\infty, \infty]$-valued measurable function $g$ such that $\int gd\mu$ is defined in $[-\infty, \infty]$ we have
\end{itemize}
$$\int_X g(x) \mathrm{d} \mu(x)=\int_Y \int_{f^{-1}(y)} g(x) \mu_y(\mathrm{d} x) \nu(\mathrm{d} y)$$
\end{thm}
For simplicity, we have stated the theorem for standard Borel probability spaces instead of general measure spaces. A standard Borel space is a seperable completely metrizable topological space (Polish space) with a Borel $\sigma$-algebra. An analytic Borel space is the continuous image of a standard Borel space.

We will also need the following disintegration theorem for differential forms for sufficient conditions that apply to most applications:
\begin{defn}
A \textbf{volume form} on a $n$-dimensional oriented manifold is a nowhere-vanishing differential $n$-form.
\end{defn}
\begin{thm}
(variant of \cite{MR350769}, theorem 16.24.8) Disintegration of volume forms: Let $M$, $N$ be two orientable smooth manifolds of dimension $n$ and $m$ respectively, and let $f:M \rightarrow N$ be a surjective submersion. Further, let $\nu$  and $\zeta$ be volume forms on $M$ and $N$ respectively. Then for all $y \in N$ a $(m-n)$-dimensional volume form $\nu / \zeta(y)$ is defined and integrable over $f^{-1}(y)$, the form $y \rightarrow \zeta(y) \int_{f^{-1}(y)} \nu / \zeta(y)$ is integrable over $N$, and
$$\int_M \nu = \int_{y \in N} \zeta(y) \int_{f^{-1}(y)} \nu / \zeta(y)$$
\end{thm}
\begin{proof}
The proof follows the original source except that the nowhere-vanishing property of $\nu$ and $\zeta$ guarantees existence of the quotient forms everywhere and since $\nu / \zeta(y)$ is given as the quotient of two nowhere-vanishing forms it is also nowhere-vanishing, hence a volume form.
\end{proof}

\subsection{Orbit Regularity}
For the proof, we need a special structure on the space $\mathcal{Y}^n$:
\begin{defn}
We call a standard Borel probability space $(X, \mathfrak{T}, \Sigma, \mu)$ orbit-regular under the action of a locally compact $\sigma$-compact group $G$ iff 
\begin{enumerate}
    \item there exists a union of a subset of the partition $\left\{G x \mid x \in X\right\}$ called $X_0$ such that $X_0$ is closed in $X$, $\mu(X_0)=0$ and the action of $G$ on $X/X_0$ is free
    \item there exists a disintegration of $\mu$ on $X \setminus X_0$ to the orbits $Y$, the densities $\mu_y$ and the left-invariant measures on the preimage of $y$ are equivalent (mutually absolutely continuous).
\end{enumerate}
\end{defn}
The second condition in our definition can be difficult to check. Therefore, we provide the following sufficient conditions for the smooth setting:
\begin{prop}
\label{prop:orb-reg-1}
(Sufficient conditions for orbit regularity 1) A standard Borel probability space $(X, \mathfrak{T}, \Sigma, \mu)$ is orbit-regular under the action of a \textbf{Lie group} $G$ if 
\begin{enumerate}
    \item $X$ is an oriented smooth manifold,
    \item  $\mu$ is given by a nowhere-vanishing volume form
    \item there exists a union of a subset of the partition $\left\{G x \mid x \in X\right\}$ called $X_0$ such that $X_0$ is closed in $X$, $\mu(X_0)=0$ and the action of $G$ on $X/X_0$ is free and smooth
    \item the topological space $(X \setminus X_0) / G$ is a manifold that can be equipped with a smooth structure such that the quotient map is a submersion
\end{enumerate}
\end{prop}
\begin{proof}
Follows from the theorem on disintegration of volume forms, the left Haar measure being a volume form, the diffeomorphism of $G$ to each orbit, and the measure equivalence of volume forms.
\end{proof}
\begin{prop} \label{prop:orb-reg-2}
(Sufficient conditions for orbit regularity 2) A standard Borel probability space $(X, \mathfrak{T}, \Sigma, \mu)$ is orbit-regular under the action of a \textbf{Lie group} $G$ if 
\begin{enumerate}
    \item $X$ is an oriented smooth manifold,
    \item $\mu$ is given by a nowhere-vanishing volume form
    \item there exists a union of a finite subset of the partition $\left\{G x \mid x \in X\right\}$ called $X_0$ such that $\mu(X_0)=0$ and the action of $G$ on $X/X_0$ is free, smooth and proper
\end{enumerate}
\end{prop}
\begin{proof}
Follows from the quotient manifold theorem and the previous proposition. Furthermore, a proper group action guarantees closed orbits.
\end{proof}
The second set of sufficient conditions is stronger than the first.

\subsection{Statement of the Theorem}
\begin{thm}
Let $\mathcal{Y} = (Y, \mathfrak{T}, \Sigma)$ be a standard Borel space with a $\sigma$-finite measure $\eta$ such that $\forall \theta \in \Theta: \mu(\mathbf{y}_{1:n}\mid \theta) \sim \eta^n \textrm{ and } \mu(\mathbf{y}^*_{1:m}\mid \mathbf{y}_{1:n}, \theta) \sim \eta^m$ (mutually absolutely continuous) and $G$ be a locally compact $\sigma$-compact group with group actions on $\Theta$, $\mathcal{Y}^n$ and $\mathcal{Y}^m$ (conditional on $\mathbf{y}_{1:n}$) such that:
$$\forall g \in G, \theta \in \Theta, A \in B(\mathcal{Y}^n): p(A \mid \theta) = p(gA \mid g\theta)$$
$$\forall  \mathbf{y}_{1:n} \in \mathcal{Y}^n, \forall g \in G, \theta \in \Theta, A \in B(\mathcal{Y}^m): p(A \mid \mathbf{y}_{1:n}, \theta) = p(g_{\mathbf{y}_{1:n}}A \mid g\mathbf{y}_{1:n}, g\theta)$$
We additionally assume that all group actions are Borel-measurable and continuous for any given $g \in G$ (and $\mathbf{y}_{1:n} \in \mathcal{Y}^n$), and that $\alpha_\Theta$ is simply transitive. Finally, we assume that $\mathcal{Y}^n$ is orbit-regular with measure $\mu(\mathbf{y}_{1:n} \mid \theta)$ for any $\theta \in \Theta$ under the action of $G$. 

\begin{enumerate}
    \item Then, the predictive risk (\ref{eq:risk}) is constant in $\theta^*$ for any invariant predictive procedure.
    \item (\ref{eq:maximin}) is solved by the posterior predictive distribution with the right-invariant (under the action of $G$, unique up to a constant factor) prior $\mu_R(\theta)$
$$\frac{\nu(\mathbf{y}^*_{1:m} \mid \mathbf{y}_{1:n})}{d\eta^m} \propto \int_{\Theta} \frac{d\mu(\mathbf{y}^*_{1:m} \mid \mathbf{y}_{1:n}, \theta)}{d\eta^m}  \frac{d\mu(\mathbf{y}_{1:n} \mid \theta)}{d\eta^n}  d\mu_R(\theta)$$
if
    \begin{enumerate}
        \item this defines a valid predictive procedure, meaning the double integral
$$\int_{\mathcal{Y}^m} \int_{\Theta} \frac{d\mu(\mathbf{y}^*_{1:m} \mid \mathbf{y}_{1:n}, \theta)}{d\eta^m}  \frac{d\mu(\mathbf{y}_{1:n} \mid \theta)}{d\eta^n}  d\mu_R(\theta) d\eta^m(\mathbf{y}^*_{1:m})$$
converges for $\eta^n$-almost all $ \in \mathcal{Y}^n$,
        \item and the predictive procedure $\nu(\mathbf{y}^*_{1:m} \mid \mathbf{y}_{1:n}): \mathcal{Y}^n  \rightarrow \mathcal{P}(\mathcal{Y}^m)$ is restricted to invariant predictive procedures.
    \end{enumerate}
    \item Additionally, if $G$ is amenable and for all $\theta \in \Theta$, $\frac{d\mu(y \mid \theta)}{d\eta}$ is continuous in $y$, optimality holds without the restriction to invariant predictive procedures.
\end{enumerate}
\end{thm}
\begin{proof}
Consider the predictive risk for a given $\theta^*$ and an invariant predictive procedure $\nu$:
$$\int_{\mathcal{Y}^n} D_{KL}(\mu(\mathbf{y}^*_{1:m} \mid \mathbf{y}_{1:n}, \theta^*) \parallel \nu(\mathbf{y}^*_{1:m} \mid \mathbf{y}_{1:n})) \mu(\mathbf{y}_{1:n} \mid \theta^*)$$
We assume $\nu(\mathbf{y}^*_{1:m} \mid \mathbf{y}_{1:n}) \ll \mu(\mathbf{y}^*_{1:m} \mid \mathbf{y}_{1:n}, \theta^*)$ $\mu(\mathbf{y}_{1:n} \mid \theta)$-almost surely since otherwise the risk would be infinite. By definition, KL-divergence is then equal to
$$= -\int_{\mathcal{Y}^n} \int_{\mathcal{Y}^m} \log\left(\frac{d\nu(\mathbf{y}^*_{1:m} \mid \mathbf{y}_{1:n})}{d\mu(\mathbf{y}^*_{1:m} \mid \mathbf{y}_{1:n}, \theta^*)}\right) d\mu(\mathbf{y}^*_{1:m} \mid \mathbf{y}_{1:n}, \theta^*) \mu(\mathbf{y}_{1:n} \mid  \theta^*)$$
Since $\mu(\mathbf{y}^*_{1:m} \mid \mathbf{y}_{1:n}, \theta^*) \sim \eta$ and hence $\nu(\mathbf{y}^*_{1:m} \mid \mathbf{y}_{1:n}) \ll \mu(\mathbf{y}^*_{1:m} \mid \mathbf{y}_{1:n}, \theta^*) \ll \eta^m$ we can write this as
$$= -\int_{\mathcal{Y}^n} \int_{\mathcal{Y}^m} \log\left(\frac{d\nu(\mathbf{y}^*_{1:m} \mid \mathbf{y}_{1:n})}{d\eta^m}\right) - \log\left(\frac{d\mu(\mathbf{y}^*_{1:m} \mid \mathbf{y}_{1:n}, \theta^*)}{d\eta^m}\right) \, d\mu(\mathbf{y}^*_{1:m} \mid \mathbf{y}_{1:n}, \theta^*) \mu(\mathbf{y}_{1:n} \mid \theta^*)$$

Now we want to disintegrate the integral over $\mathcal{Y}^n$ using the orbit-regularity of the space. Let $\mathcal{Y}_0^n$ be the subset on which $G$ does not act transitively. Since it has measure 0 under any measure $\mu(\mathbf{y}_{1:n} \mid \theta^*)$ we can write the integral as
$$=-\int_{\mathcal{Y}^n \setminus \mathcal{Y}_0^n} \int_{\mathcal{Y}^m} \log\left(\frac{d\nu(\mathbf{y}^*_{1:m} \mid \mathbf{y}_{1:n})}{d\eta^m}\right)- \log\left(\frac{d\mu(\mathbf{y}^*_{1:m} \mid \mathbf{y}_{1:n}, \theta^*)}{d\eta^m}\right) \, d\mu(\mathbf{y}^*_{1:m} \mid \mathbf{y}_{1:n}, \theta^*) \mu(\mathbf{y}_{1:n} \mid \theta^*)$$
Since $\mathcal{Y}_0^n$ is assumed to be closed, $\mathcal{Y}^n \setminus \mathcal{Y}_0^n$ is still a standard Borel space.

Consider its orbit space projection under the group action $\pi: \mathcal{Y} \rightarrow \mathcal{Y} / G$. This induces the orbit space  $((\mathcal{Y}^n \setminus \mathcal{Y}_0^n)/G, \mu(o \mid \theta^*))$ as a new Borel probability space with $\pi$ as an inverse-measure-preserving function. Since the orbit space is the continuous image of the standard Borel space $\mathcal{Y}^n \setminus \mathcal{Y}_0^n$ under the quotient map, it is an analytic Borel space. Now by the disintegration theorem there exist probability measures $\left\langle\mu_o\right\rangle_{o \in \mathcal{Y}/G}$ such that our optimization objective becomes
\begin{equation}
\label{eq:obj-disintegrated}
=-\int_{(\mathcal{Y}^n \setminus \mathcal{Y}_0^n) / G} \int_{\phi^{-1}(o)} \int_{\mathcal{Y}^m} \log\left(\frac{d\nu(\mathbf{y}^*_{1:m} \mid \mathbf{y}_{1:n})}{d\eta^m}\right) - \log\left(\frac{d\mu(\mathbf{y}^*_{1:m} \mid \mathbf{y}_{1:n}, \theta^*)}{d\eta^m}\right) 
\end{equation}
$$d\mu(\mathbf{y}^*_{1:m} \mid \mathbf{y}_{1:n}, \theta^*) d\mu_o(\mathbf{y}_{1:n} \mid \theta^*) d\rho(o \mid \theta^*)$$
Orbits are invariant under the group action $\rho(o \mid \theta^*) = \rho(o \mid g\theta^*)$. Since $G$ acts simply transitive on $\Theta$, $\rho(o \mid \theta^*)$ does not depend on $\theta^*$. Since we have assumed that the action of $G$ on $\mathcal{Y}$ is free, it is simply transitive on every orbit. Choosing a representative from each orbit (using the axiom of choice), we can represent every element $x$ within an orbit as the unique group element transforming the representative into $x$. This allows us to write the integral as
$$= -\int_{(\mathcal{Y}^n \setminus \mathcal{Y}_0^n) / G} \int_{G} \int_{\mathcal{Y}^m} \log\left(\frac{d\nu(\mathbf{y}^*_{1:m} \mid o, g)}{d\eta^m}\right) - \log\left(\frac{d\mu(\mathbf{y}^*_{1:m} \mid o, g, \theta^*)}{d\eta^m}\right) \, d\mu(\mathbf{y}^*_{1:m} \mid o, g, \theta^*) d\mu_o(g \mid \theta^*) d\rho(o)$$

Since $G$ is assumed to be $\sigma$-compact, its left Haar measure is $\sigma$-finite. We choose a normalization of the left Haar measure that we use for every orbit. Our orbit-regularity assumption guarantees that $\forall o: \mu_o(g \mid \theta^*) \ll d\mu_L(g)$, allowing us to apply the Radon-Nikodym theorem:
$$= -\int_{(\mathcal{Y}^n \setminus \mathcal{Y}_0^n) / G} \int_{G} \int_{\mathcal{Y}^m} \log\left(\frac{d\nu(\mathbf{y}^*_{1:m} \mid o, g)}{d\eta^m}\right) - \log\left(\frac{d\mu(\mathbf{y}^*_{1:m} \mid o, g, \theta^*)}{d\eta^m}\right) $$
$$d\mu(\mathbf{y}^*_{1:m} \mid o, g, \theta^*) \frac{d\mu_o(g \mid \theta^*)}{d\mu_L} d\mu_L(g) d\rho(o)$$
For any $g_1 \in G$:
$$=-\int_{(\mathcal{Y}^n \setminus \mathcal{Y}_0^n) / G} \int_{G} \int_{\mathcal{Y}^m} \log\left(\frac{d\nu(\mathbf{y}^*_{1:m} \mid o, g g_1^{-1} g_1)}{d\eta^m}\right) - \log\left(\frac{d\mu(\mathbf{y}^*_{1:m} \mid o, g, \theta^*)}{d\eta^m}\right)$$
$$ d\mu(\mathbf{y}^*_{1:m} \mid o, g, \theta^*) \frac{d\mu_o(g g_1^{-1} g_1 \mid \theta^*)}{d\mu_L} d\mu_L(g) d\rho(o)$$
By disintegrating $\mu(h \mathbf{y}_{1:n} \mid h\theta^*)$ like $\mu(\mathbf{y}_{1:n} \mid \theta^*)$ we see that it can be disintegrated with orbit densities $\mu_o(hg \mid h\theta^*)$ and the same density on $\mathcal{Y}^n/G$. Hence, $\mu_o(hg \mid h\theta^*) = \mu_o(g \mid \theta^*)$. Due to the left-invariance of $\mu_L(g)$ we get:
$$= - \int_{(\mathcal{Y}^n \setminus \mathcal{Y}_0^n) / G} \int_{G} \int_{\mathcal{Y}^m} \log\left(\frac{d\nu_o(g_1 g^{-1} \mathbf{y}^*_{1:m} \mid g_1)}{d\eta^m}\right) - \log\left(\frac{d\mu(g_1 g^{-1}\mathbf{y}^*_{1:m} \mid o, g_1, g_1 g^{-1}\theta^*)}{d\eta^m}\right) $$
$$ d\mu(\mathbf{y}^*_{1:m} \mid o, g, \theta^*) \frac{d\mu_o(g_1 \mid g_1 g^{-1} \theta^*)}{d\mu_L} d\mu_L(g) d\rho(o)$$
The Jacobian determinants $\left| \det J^{\eta^m}_{g_1 g^{-1}}(\mathbf{y}^*_{1:m})\right|$ cancel out between the log-terms. Since $\mu(\mathbf{y}_{1:m}, \mathbf{y}_{1:n} \mid \theta^*) \sim \eta^{n+m}$ and $\mu(\mathbf{y}_{1:n} \mid \theta^*) \sim \eta^n$ we have $\mu(\mathbf{y}_{1:m} \mid \mathbf{y}_{1:n}, \theta^*) \sim \eta^m$. Next we use the invariance of $\mu(\mathbf{y}^*_{1:m} \mid \theta^*)$ to arrive at
$$-\int_{(\mathcal{Y}^n \setminus \mathcal{Y}_0^n) / G}\int_{G} \int_{\mathcal{Y}^m} \left( \log\left(\frac{d\nu_o(g_1 g^{-1} \mathbf{y}^*_{1:m} \mid g_1)}{d\eta^m}\right)- \log\left(\frac{d\mu(g_1 g^{-1}\mathbf{y}^*_{1:m} \mid o, g_1, g_1 g^{-1}\theta^*)}{d\eta^m}\right)\right)$$
$$\left| \det J^{\eta^m}_{g_1 g^{-1}}(\mathbf{y}^*_{1:m}) \right| \frac{d\mu(g_1 g^{-1}\mathbf{y}^*_{1:m} \mid o, g_1, g_1 g^{-1}\theta^*)}{d\eta^m} d\eta^m \frac{d\mu_o(g_1 \mid g_1 g^{-1} \theta^*)}{d\mu_L} d\mu_L(g) d\rho(o)$$
Substituting $\mathbf{y}^*_{1:m}$ with $g g_1^{-1} \mathbf{y}^*_{1:m}$ the Jacobian cancels out:
$$= -\int_{(\mathcal{Y}^n \setminus \mathcal{Y}_0^n) / G} \int_{G} \int_{\mathcal{Y}^m} \log\left(\frac{d\nu_o(\mathbf{y}^*_{1:m} \mid g_1)}{d\eta^m}\right) - \log\left(\frac{d\mu(\mathbf{y}^*_{1:m} \mid o, g_1, g_1 g^{-1}\theta^*)}{d\eta^m}\right)$$
$$ \frac{d\mu(\mathbf{y}^*_{1:m} \mid o, g_1, g_1 g^{-1}\theta^*)}{d\eta^m} d\eta^m \frac{d\mu_o(g_1 \mid g_1 g^{-1} \theta^*)}{d\mu_L}d\mu_L(g)d\rho(o)$$
Due to the relationship between the left and right Haar measure, we get:
$$= -\Delta(g_1) \int_{(\mathcal{Y}^n \setminus \mathcal{Y}_0^n) / G}  \int_{\mathcal{Y}^m} \int_{G} \log\left(\frac{d\nu_o(\mathbf{y}^*_{1:m} \mid g_1)}{d\eta^m}\right) - \log\left(\frac{d\mu(\mathbf{y}^*_{1:m} \mid o, g_1, g\theta^*)}{d\eta^m}\right)$$
$$ \frac{d\mu(\mathbf{y}^*_{1:m} \mid o, g_1, g\theta^*)}{d\eta^m} \frac{d\mu_o(g_1 \mid g\theta^*)}{d\mu_L}d\mu_R(g)d\eta^m  d\rho(o)$$
Since $G$ acts simply transitively on $\Theta$:
$$
\label{eq:ind-of-theta_star}
= -\Delta(g_1) \int_{(\mathcal{Y}^n \setminus \mathcal{Y}_0^n) / G}  \int_{\mathcal{Y}^m} \int_{\Theta} \log\left(\frac{d\nu_o(\mathbf{y}^*_{1:m} \mid g_1)}{d\eta^m}\right) - \log\left(\frac{d\mu(\mathbf{y}^*_{1:m} \mid o, g_1, \theta)}{d\eta^m}\right)$$
$$ \frac{d\mu(\mathbf{y}^*_{1:m} \mid o, g_1, \theta)}{d\eta^m} \frac{d\mu_o(g_1 \mid \theta)}{d\mu_L}  d\mu_R(\theta)d\eta^m d\rho(o)$$
where $\mu_R(\theta)$ is the unique up to a constant factor right-invariant measure on $\Theta$ under the simply transitive group action of $G$. This shows that the predictive risk of \textbf{any} invariant predictive procedure is constant in $\theta^*$.

We now want to find the invariant predictive procedure $\nu$ minimizing this predictive risk. Hence, we drop the term that does not depend on $\nu$:
$$\Delta(g_1) \int_{(\mathcal{Y}^n \setminus \mathcal{Y}_0^n) / G} \left( -\int_{\mathcal{Y}^m}\log\left(\frac{d\nu_o(\mathbf{y}^*_{1:m} \mid g_1)}{d\eta^m}\right) \int_{\Theta} \frac{d\mu(\mathbf{y}^*_{1:m} \mid o, g_1, \theta)}{d\eta^m}  \frac{d\mu_o(g_1 \mid \theta)}{d\mu_L}  d\mu_R(\theta)d\eta^m \right) d\rho(o)$$

Consider the measures $\mu(o) \times \mu_L(g)$ and $\mu(o) \times \mu_o(g_1 \mid \theta)$ on $(\mathcal{Y}^n \setminus \mathcal{Y}_0^n)/G \times G$. Then
$$\frac{d\mu_o}{d\mu_L}(g_1 \mid \theta) = \frac{d\rho \times \mu_o}{d\rho \times \mu_L}(o, g_1 \mid \theta)$$
We have $d\mu(o) \times \mu_L(g) \sim p(\mathbf{y}_{1:n} \mid \theta) \sim \eta^n$. Pulling out the Radon-Nikodym derivative $\frac{d\eta^n}{d\rho \times \mu_L}(o, g_1)$ we have:
$$\Delta(g_1) \int_{(\mathcal{Y}^n \setminus \mathcal{Y}_0^n) / G} \frac{d\eta^n}{d\rho \times \mu_L}(o, g_1) \left( -\int_{\mathcal{Y}^m}\log\left(\frac{d\nu_o(\mathbf{y}^*_{1:m} \mid g_1)}{d\eta^m}\right)\right.$$
$$\left.\int_{\Theta} \frac{d\mu(\mathbf{y}^*_{1:m} \mid o, g_1, \theta)}{d\eta^m}  \frac{d\mu(o, g_1 \mid \theta)}{d\eta^n}  d\mu_R(\theta)d\eta^m \right) d\rho(o)$$

The expression
$$\int_{\Theta} \frac{d\mu(\mathbf{y}^*_{1:m} \mid o, g_1, \theta)}{d\eta^m}  \frac{d\mu(o, g_1 \mid \theta)}{d\eta^n}  d\mu_R(\theta)$$
can now be recognized as the standard formulation of the Bayesian posterior predictive distribution, with a right-invariant prior. $(o, g_1)$ can be recognized as our observation. While the predictive risk is constant in $g_1$, for $o$ we need to integrate over $\rho(o)$, weighted by $\frac{d\eta^n}{d\rho \times \mu_L}(o, g_1)$.

Since we assume the integral to converge and then also be integrable in $\mathbf{y}^*_{1:m}$ in the statement of our theorem, we are minimizing the cross-entropy to this probability density over $\mathbf{y}^*_{1:m}$, scaled by a constant factor. The cross-entropy $H(p, q)$ is minimized in $p$ by $p:=q$. Hence, the optimal invariant predictive procedure is given by:
$$\frac{\nu(\mathbf{y}^*_{1:m} \mid \mathbf{y}_{1:n})}{d\eta^m} \propto \int_{G} \frac{d\mu(\mathbf{y}^*_{1:m} \mid \mathbf{y}_{1:n}, \theta)}{d\eta^m}  \frac{d\mu(\mathbf{y}_{1:n} \mid \theta)}{d\eta^n}  d\mu_R(\theta)$$

We can verify that for any $h \in G$ we have the desired invariance:
$$\nu(h \mathbf{y}^*_{1:m} \mid h \mathbf{y}_{1:n}) \propto \int_{G} \mu(h \mathbf{y}^*_{1:m} \mid \mathbf{y}_{1:n}, \theta) \frac{d\mu(h \mathbf{y}_{1:n} \mid \theta)}{d\eta^n}  d\mu_R(\theta)$$
$$\propto \int_{G} \mu(\mathbf{y}^*_{1:m} \mid \mathbf{y}_{1:n}, h^{-1} \theta) \frac{d\mu(g \mid h^{-1} \theta)}{d\eta^n}  d\mu_R(\theta) \propto \nu(\mathbf{y}^*_{1:m} \mid \mathbf{y}_{1:n}) $$
If $f$ is our risk function we have shown that $\argmin_\nu f(\theta, \nu)$ is constant in $\theta$, and hence have $\argmin_\nu \max_\theta = \argmin_\nu f(\theta, \nu)$. Therefore, this choice of $\nu$ minimizes the worst-case predictive risk (\ref{eq:maximin}). Additionally, the predictive risk of any invariant predictive procedure is affine in the weighted cross-entropy in expectation over $\mu(o)$ (and any choice of $g_1$) to this optimal predictive procedure. 

Finally, Liang \& Barron's theorem \cite{liang2004exact} implies that the best invariant predictive procedure is the best predictive procedure overall if $G$ is amenable. While they assume a simpler invariance structure where the group action on $\mathcal{Y}^m$ and $\mathcal{Y}^n$ are independently well-defined, the argument in their proof still holds with our non-separable invariance structure, where the action on $\mathcal{Y}^m$ may depend on $\mathcal{Y}^n$.
\end{proof}
In addition to proving that the best invariant predictive procedure is the best predictive procedure, Liang \& Barron \cite{liang2004exact} also provide an intuition on why the statement 1 and 2 of our theorem should be true in Appendix D of their paper. However, they do not state precise assumptions (like freeness of the group action) or provide a proof beyond what Komaki \cite{komaki2002bayesian} has shown.

Note that unlike Bayesian prediction with a proper prior, we may need multiple observations for the posterior predictive integral to converge, allowing us to make a prediction. Even then, we can usually only guarantee that the integral converges $\eta^n$-almost always instead of always.

\section{Application to the Multivariate Normal Distribution}
Our theorem now allows us to provide a provably maximin-optimal predictive procedure for the multivariate normal distribution.

\begin{lemma}
$(\mathbb{R}^d)^n$  is orbit-regular under the action on $G_\mathcal{N}$ if $n>d$.
\end{lemma}
\begin{proof}
We will use Proposition \ref{prop:orb-reg-1}.

1. $(\mathbb{R}^d)^n$ is a smooth manifold.

2. The multivariate normal distribution defines a probability density that is positive everywhere with respect to Lebesgue measure, and can hence be represented as a volume form.

3. An affine transformation on $\mathbb{R}^d$ is completely determined by how it acts on $d+1$ points in general position.  $n$ points are in general position if no subset of $d+1$ points is affinely dependent. Affine dependence of the points $\mathbf{y}_1, ..., \mathbf{y}_{d+1}$ can be checked by the condition:
$$\det \begin{pmatrix}
1 & 1 & \cdots & 1 \\
y_{11} & y_{21} & \cdots & y_{(d+1)1} \\
y_{12} & y_{22} & \cdots & y_{(d+1)2} \\
\vdots & \vdots & \ddots & \vdots \\
y_{1d} & y_{2d} & \cdots & y_{(d+1)d}
\end{pmatrix} = 0$$
As a preimage of the closed set $\{0\}$ by a continuous function this is a closed subset of $(\mathbb{R}^d)^n$, and $\mu(\mathbf{y}_{1:n} \mid \theta)$ assign $0$ measure for any $\theta \in \Theta$. There are $n \choose d+1$ such submanifolds of sets of points in special positions, together they form the set $S$. The finite union of closed sets of measure 0 is still a closed set of measure 0. Additionally, general position of a set of points is preserved under the group action, allowing us to represent the excluded special position set $X_0$ as the union of orbits. With this exclusion, there is also at most one element of the group transforming  any set of points into any other, guaranteeing freeness of the group action. 

4. If we have a set of $n > d$ points in $\mathbb{R}^d$ in general position, we have the sample mean and positive definite sample covariance matrix
$$\mathbf{\bar{x}}=\frac{1}{n}\sum_{i=1}^{n}\mathbf{x}_i, \quad \bar{\Sigma} = {1 \over {n-1}}\sum_{i=1}^n (\mathbf{x}_i-\mathbf{\bar{x}}) (\mathbf{x}_i-\mathbf{\bar{x}})^T$$
Let $N$ be the set of sets of $n$ points in general position with $\mathbf{\bar{x}} = \mathbf{0}$ and $\bar{\Sigma} = I$. Since $(\mathbf{\bar{x}} = \mathbf{0}, \bar{\Sigma} = I)$ is a regular value (see appendix \ref{sec:regular-value-proof}) the preimage theorem guarantees that $N$ is a smooth submanifold. Every orbit of the group action $V\mathbf{x} + \mathbf{m}$ applied element-wise to a set of points $(\mathbb{R}^d)^n$ in general position contains exactly one set of points in $N$. Therefore we can identify the orbits with the map
$$m: (\mathbf{x}_1, .., \mathbf{x}_n) \rightarrow (\bar{\Sigma}^{-1/2}(\mathbf{x}_1 - \mathbf{\bar{x}}), ..., \bar{\Sigma}^{-1/2}(\mathbf{x}_n - \mathbf{\bar{x}}))$$
where $\bar{\Sigma}^{-1/2}$ is the inverse of the upper triangular Cholesky decomposition of the sample covariance matrix. Furthermore, consider the map $\phi: (\mathbb{R}^d)^n \setminus X_0 \rightarrow N \times G$:
$$\phi(X) = \left(m(X), \left(\mathbf{\bar{x}}, \bar{\Sigma}^{1/2}\right)\right)$$
This map is smooth since the Cholesky decomposition and all other operations used are smooth, injective and surjective. Furthermore, we have the inverse map: $\phi^{-1}: N \times G \rightarrow (\mathbb{R}^d)^n \setminus X_0$:
$$\phi^{-1}(Y, (m, V)) = \left\{VY_i + m\right\}_{i=1}^n$$
which is also smooth. Therefore $\phi$ is a diffeomorphism. Since the projection onto a factor of a product manifold (like $N \times G$) is always a submersion, this further implies that equipping $((\mathbb{R}^d)^n \setminus X_0)/G$ with the smooth structure of $N$ makes the quotient map a submersion. 
\end{proof}

\begin{prop}
(\cite{hewitt1979abstract}, 15.28) The right Haar measure of the group $T(n, \mathbb{R})$ of real upper triangular $n \times n$ matrices is given by $\prod_{i=1}^n (|A_{i,i}|)^{-i} dA_{1, 1} dA_{1, 2}, ..., dA_{n, n}$.
\end{prop}
\begin{coro}
The right Haar measure of the group $T^+(n, \mathbb{R})$ of real upper triangular $n \times n$ matrices with a positive diagonal is given by 
$$\prod_{i=1}^n (A_{i,i})^{-i} dA_{1, 1} dA_{1, 2}, ..., dA_{n, n}$$
\end{coro}
\begin{proof}
Since the subgroup $T^+(n, \mathbb{R})$ of $T(n, \mathbb{R})$ has positive measure assigned by the right Haar measure of $T(n, \mathbb{R})$, its Haar measure is simply the restriction.
\end{proof}
\begin{prop} \label{prop:right-haar}
(\cite{nachbin1965haar}, Proposition 28) The right Haar measure of the semidirect product of two locally compact groups is the product of their right Haar measures.
\end{prop}
Now we have the following corollary of the main result:
\begin{coro}
An optimal predictive procedure for the $d$-dimensional multivariate normal distribution in the sense of (\ref{eq:maximin}) when having observed $n > d$ samples and wanting to predict the next $m$ samples is given by the Bayesian posterior predictive with improper prior
$$\mu_R(U, \mu) \propto \prod_{i=1}^d (U_{i,i})^{-i} dU_{1, 1} dU_{1, 2}, ..., dU_{d, d} d\mu$$
Here the normal distribution is parametrized in terms of the mean $\mu$ and the upper triangular Cholesky decomposition of the covariance matrix $U$. Additionally, this predictive procedure is invariant under $G_\mathcal{N}$.
\end{coro}
\begin{proof}
The posterior is almost surely proper for this prior after $n>d$ observations as shown in appendix \ref{sec:posterior-propriety}. Hence, the posterior predictive with the right invariant prior defines a valid predictive procedure. $\mathbb{R}^d$ is a standard Borel space under the Euclidean topology and orbit-regular. We choose the standard Borel measure of $\mathbb{R}^d$ as $\eta$. Furthermore, $(\mathbb{R}^d)^n$ is orbit-regular under $G_\mathcal{N}$.  $G_\mathcal{N}$ is a semidirect product of $T^+(d, \mathbb{R})$ and $(\mathbb{R}^d, +)$. Since the latter is unimodular, the given prior is a right-invariant measure on $\Theta$. Since $G_\mathcal{N}$ is also amenable and the multivariate normal likelihood is continuous with respect to $\eta$, optimality among all predictive procedures is guaranteed.
\end{proof}

\subsection{Uniqueness}
We have previously stated that the right-invariant measure is unique up to a constant factor given $G$ and its action. We have even shown that the choice of $G$ is unique for the multivariate normal distribution in Proposition \ref{prop:unique-group}. However, the \textbf{group action} is not unique: Different coordinate charts of $\mathbb{R}^d$ induces by different bases and origins lead to different group actions on the space of multivariate normal distributions. Changing the origin has no effect on the right-invariant prior since our prior is constant in the mean. However, changing the basis, for example permuting the basis vectors, can lead to entirely different right-invariant priors (over the covariance matrix). The right-invariant prior density will be the same, however it is specified with respect to a different coordinate chart, leading to a different measure on the space of multivariate normal distributions. 

To gain a better understanding about how the different right-invariant priors are related and show they can indeed be different, we will convert the posterior predictive in the standard coordinate to the coordinate chart with the same origin but the basis $B$. This is the right-invariant prior posterior predictive in the standard basis:
$$\int_{\mathbb{R}^d} \int_{T^+(d, \mathbb{R})} p(\mathbf{y}_1^*, \dots, \mathbf{y}_m^*, \mathbf{y}_1, \dots, \mathbf{y}_n \mid U, \mu) \prod_{i=1}^d (U_{i,i})^{-i} dU_{1, 1} dU_{1, 2}, \dots, dU_{d, d} d\mu$$
Due to the nontrivial behavior of the Cholesky decomposition under a change of basis, we first perform a change of variables to the standard coordinate system of symmetric matrices, based on the upper half of a matrix (including its diagonal). For the upper-triangular Cholesky decomposition this results in (similar to \cite{deemer1951jacobians} Theorem 4.1 but for upper triangular matrices):
$$\int_{\mathbb{R}^d} \int_{\Sigma \succ 0} p(\mathbf{y}_1^*, \dots, \mathbf{y}_m^*, \mathbf{y}_1, \dots, \mathbf{y}_n \mid \Sigma, \mu) \prod_{i=1}^d (\textrm{chol}(\Sigma)_{i,i})^{-2i} d\Sigma_{1, 1} d\Sigma_{1, 2}, \dots, d\Sigma_{d, d} d\mu$$
Now we substitute the integrands $\Sigma' = B \Sigma B^T$ and $\mu' = B \mu$:
$$\int_{\mathbb{R}^d} \int_{\Sigma \succ 0}  p(\mathbf{y}_1^*, \dots, \mathbf{y}_m^*,  \mathbf{y}_1, \dots, \mathbf{y}_n \mid B^{-1} \Sigma' B^{-T}, B^{-1}\mu') \prod_{i=1}^d (\textrm{chol}(B^{-1} \Sigma' B^{-T})_{i,i})^{-2i} d\Sigma' _{1, 1} d\Sigma'_{1, 2}, \dots, d\Sigma'_{d, d} d\mu'$$
Since the transformation is linear, the determinant is only a constant factor that does not change the predictive procedure. Additionally, we represent the observations in our new coordinate chart $\mathbf{y}' = B \mathbf{y}$ and apply the linear transformation rule for multivariate normal distributions:
$$\int_{\mathbb{R}^d} \int_{\Sigma \succ 0} 
 p({\mathbf{y}_1^*}', \dots, {\mathbf{y}_m^*}', \mathbf{y}_1', \dots, \mathbf{y}_n' \mid  \Sigma', \mu') \prod_{i=1}^d (\textrm{chol}(B^{-1} \Sigma' B^{-T})_{i,i})^{-2i} d\Sigma' _{1, 1} d\Sigma'_{1, 2}, \dots, d\Sigma'_{d, d} d\mu'$$
When decomposing a covariance matrix, the squared diagonal entries of the Cholesky decomposition can be interpreted as the variance along the respective basis vectors, conditional on the previous basis vectors. \cite{579332} Therefore, our results shows that the right-invariant priors with the respect to the standard symmetric matrix measure in some coordinate chart are the conditional variances in some other basis to the power of $-i$. This also shows that the right-invariant prior and its posterior predictive are not unique.

These results are unsurprising, since invariance under any change of basis would require invariance under the full affine group $\mathrm{Aff}(d, \mathbb{R}) = (\mathbb{R}^d, +) \rtimes GL(d, \mathbb{R})$ instead of just $(\mathbb{R}^d, +) \rtimes T^+(d, \mathbb{R})$.  As Geisser argued, the posterior predictive from the independence-Jeffreys prior $\det(\Sigma)^{-\frac{d+1}{2}}$ is the optimal posterior predictive procedure that is invariant under $\mathrm{Aff}(d, \mathbb{R})$. \cite{geisser1993predictive} We will numerically compare these predictive procedures in the next section. While the right-invariant priors are not unique, the predictive procedures they induce all share the same predictive risk, as guaranteed by our theorem.

\begin{figure}[htbp]
    \centering
    \begin{subfigure}{0.25\textwidth}
        \includegraphics[width=\textwidth]{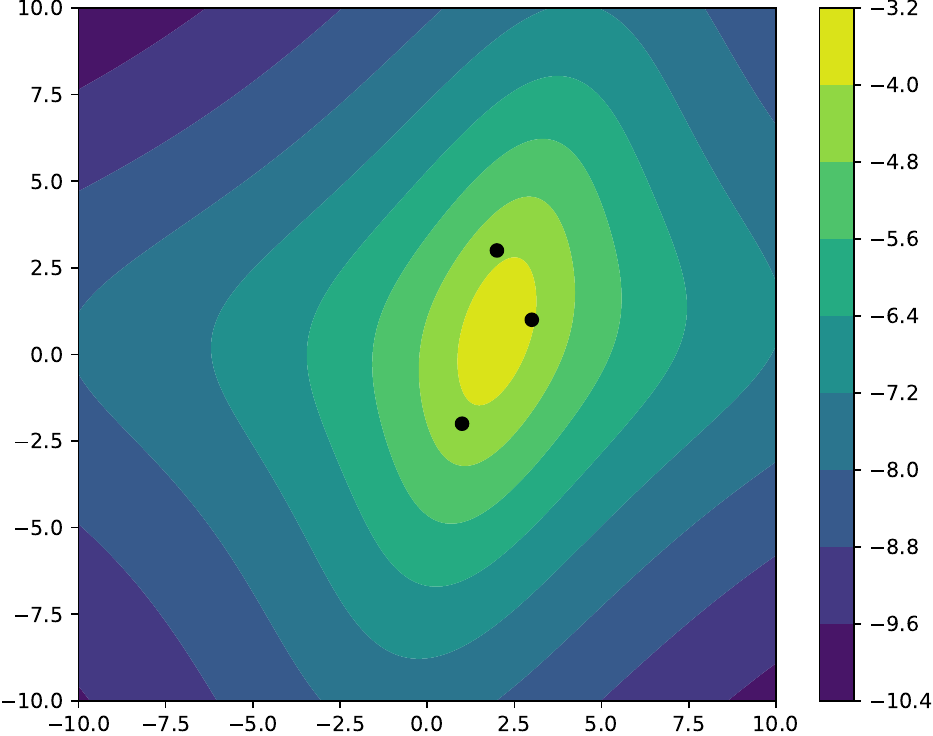}
        \caption{Right-invariant prior}
    \end{subfigure}
    \hfill 
    \begin{subfigure}{0.25\textwidth}
        \includegraphics[width=\textwidth]{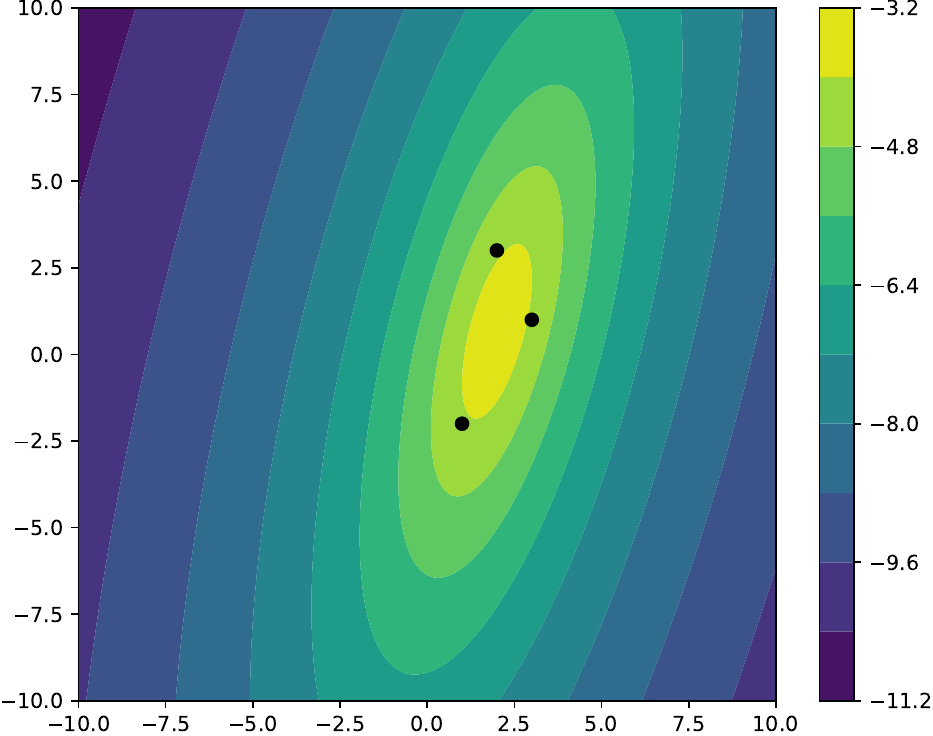}
        \caption{Independence-Jeffreys prior}
    \end{subfigure}
    \hfill 
    \begin{subfigure}{0.25\textwidth}
        \includegraphics[width=\textwidth]{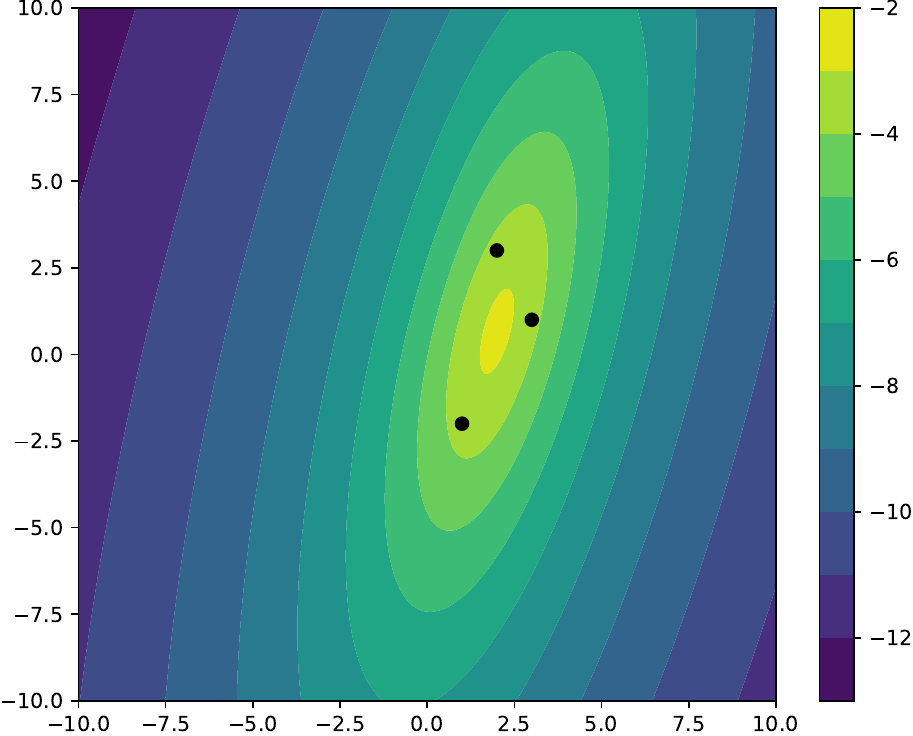}
        \caption{Jeffreys prior}
    \end{subfigure}

    \medskip 

    \begin{subfigure}{0.25\textwidth}
        \includegraphics[width=\textwidth]{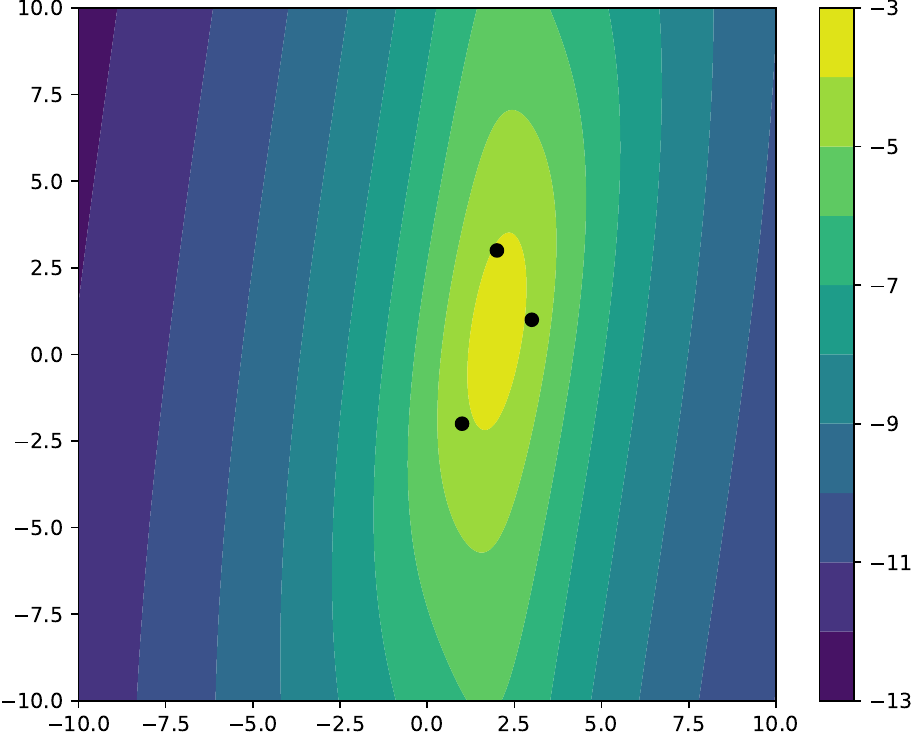}
        \caption{Right-invariant prior swapped}
    \end{subfigure}
    \hfill 
    \begin{subfigure}{0.25\textwidth}
        \includegraphics[width=\textwidth]{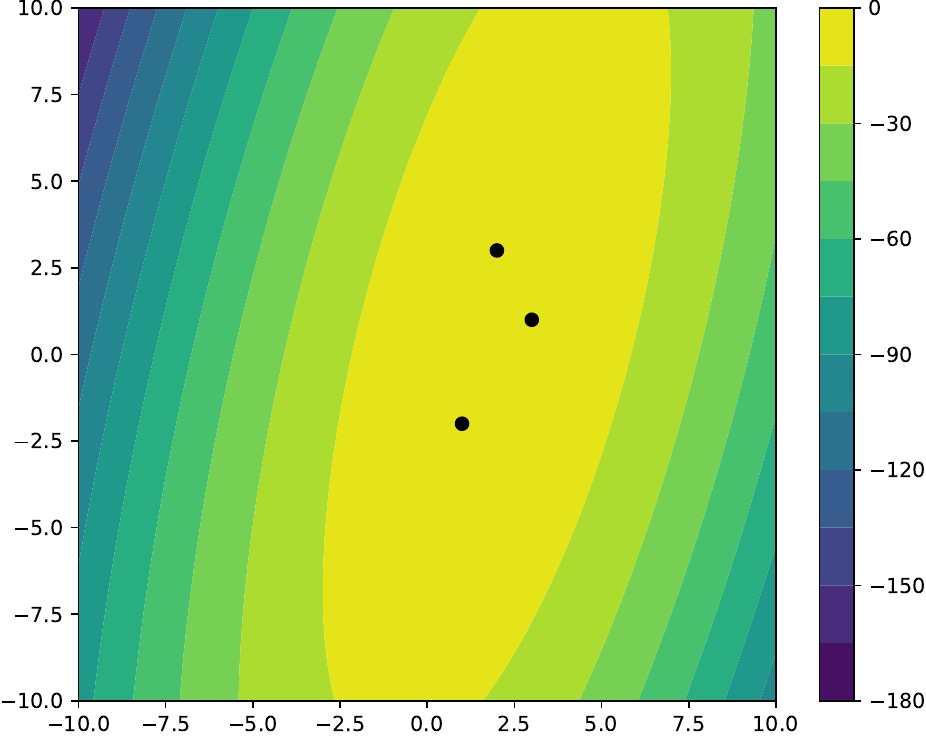}
        \caption{Unbiased plug-in}
    \end{subfigure}
    \hfill 
    \begin{subfigure}{0.25\textwidth}
        \includegraphics[width=\textwidth]{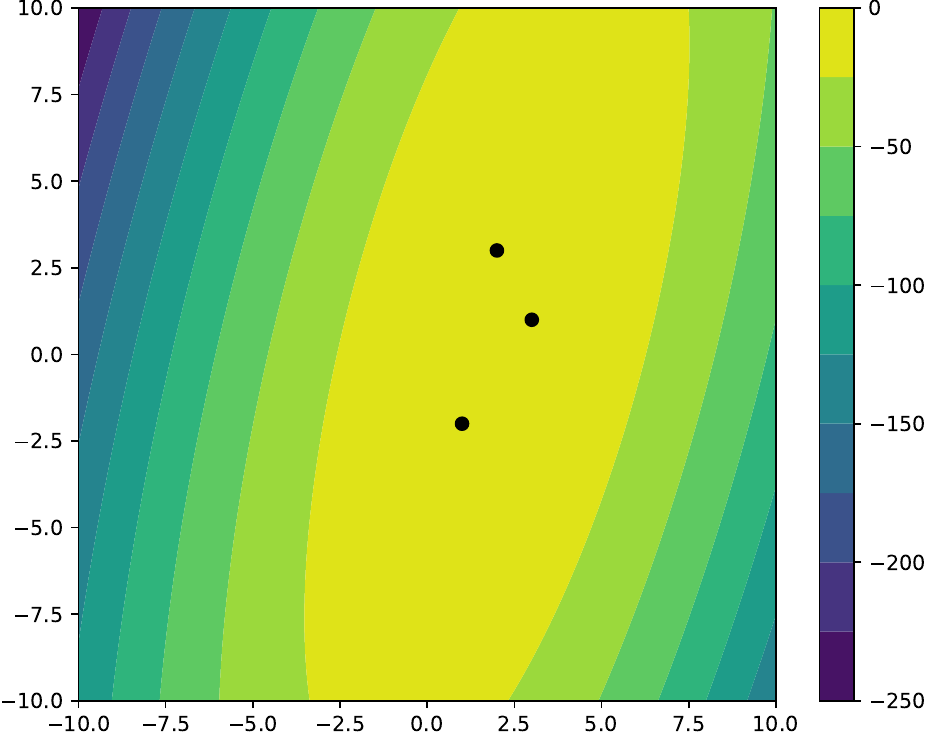}
        \caption{MLE plug-in}
    \end{subfigure}

    \caption{Log-likelihood of different predictive procedures on the bivariate normal distribution for three observations}
    \label{fig:grid}
\end{figure}

\subsection{Numerical Results}
We now want to numerically compare the performance of the posterior predictive distribution of a right-invariant prior $q_\mathrm{R}$ to four alternative invariant predictive procedures: Jeffreys prior $q_\mathrm{J}$, the independence-Jeffreys prior $q_\mathrm{IJ}$ and plug-in prediction with the maximum-likelihood estimate $q_\mathrm{MLE}$ and with unbiased sample covariance $q_\mathrm{unb}$. We choose next-sample prediction $(m=1)$ for the bivariate normal distribution since all five predictive procedures are different in this case and we have a closed-form expression for all of them. 

For all predictive procedures except $q_R$ we have closed-form expressions for any dimensionality $d$ of the normal distribution. Let us define:
$$\mathbf{\bar{x}}=\frac{1}{n}\sum_{i=1}^{n}\mathbf{x}_i, \quad S = \sum_{i=1}^n (\mathbf{x}_i-\mathbf{\bar{x}}) (\mathbf{x}_i-\mathbf{\bar{x}})^T$$
Then we have: \cite{murphy2007conjugate}
$$q_\mathrm{IJ}(\mathbf{x}^* \mid \mathbf{x}_1, ..., \mathbf{x}_n) = t_{n-d}\left(\mathbf{x}^* \left| \mathbf{\bar{x}}, \frac{(n+1)}{n(n-d)} S \right.\right)$$
$$q_\mathrm{J}(\mathbf{x}^* \mid \mathbf{x}_1, ..., \mathbf{x}_n) = t_{n-d+1}\left(\mathbf{x}^* \left| \mathbf{\bar{x}}, \frac{(n+1)}{n(n-d+1)} S\right.\right)$$
$$q_\mathrm{unb}(\mathbf{x}^* \mid \mathbf{x}_1, ..., \mathbf{x}_n) = \mathcal{N}\left(\mathbf{x}^* \left| \mathbf{\bar{x}}, \frac{1}{n-1}S\right.\right)$$
$$q_\mathrm{MLE}(\mathbf{x}^* \mid \mathbf{x}_1, ..., \mathbf{x}_n) = \mathcal{N}\left(\mathbf{x}^* \left| \mathbf{\bar{x}}, \frac{1}{n}S\right.\right)$$
The expression for $q_\mathrm{R}$ is derived for $d=2$ in appendix \ref{sec:bivn-predictive}. We use $\mathbf{u} = (x_{1,1}, ..., x_{n, 1})^T$ and $\mathbf{v} = (x_{1, 2}, ..., x_{n, 2})^T$. $G(\mathbf{v}_1, ..., \mathbf{v}_k)$ denotes the Gram matrix of the vectors $\mathbf{v}_1, ..., \mathbf{v}_k$. Then:
$$
q_R(\mathbf{x}^* \mid \mathbf{x}_1, ..., \mathbf{x}_n) = \frac{n-2}{2\pi} \frac{(n+1)^{(n-1)/2}}{n^{(n-2)/2}} \frac{\det(G( \mathbf{v}, \mathbf{1}_n))}{\det(G(\mathbf{u}, \mathbf{v}, \mathbf{1}_n))^{-\frac{n-2}{2}}}
\frac{\det\left(G\left(\begin{pmatrix}\mathbf{u} \\ x^*_1 \end{pmatrix}, \begin{pmatrix}\mathbf{v} \\ x^*_2 \end{pmatrix}, \mathbf{1}_{n+1}\right)\right)^{-\frac{n-1}{2}}}{\det\left(G\left(\begin{pmatrix}\mathbf{v} \\ x^*_2 \end{pmatrix}, \mathbf{1}_{n+1}\right)\right)}
$$
We evaluate the predictive risk of every predictive procedure for different numbers of observations $n$ using a Monte Carlo estimate using between $2^{28}$ and $2^{32}$ samples for each combination, also estimating the MC estimator's standard deviation. At least three observations are needed for the predictive procedures to become well-defined. Since all evaluated predictive procedures are invariant, their risk does not depend on $\theta^*$. We have summarized our results in Table \ref{tab:pred-risk}.

\begin{table}
    \centering
    \begin{tabular}{|c|c|c|c|c|c|c|c|c|} \hline 
         n&  3&  4&  5&  6&  7&  8&  9& 10\\ \hline 
         $q_\mathrm{R}$&  1.324&  0.839&  0.620&  0.494&  0.411
&  0.3518&  0.3078& 0.2738\\ \hline 
         $q_\mathrm{IJ}$&  1.711&  0.953&  0.673&  0.524&  0.430&  0.366
&  0.3179& 0.2816\\ \hline 
 $q_\mathrm{J}$& 2.018& 1.047& 0.719
& 0.551& 0.448
& 0.378& 0.327&0.289\\ \hline 
         $q_\mathrm{unb}$&  $(2.3 \pm 1.4) * 10^9$&  $37.0 \pm 1.4$&  3.3&  1.56&  1.000&  0.729&  0.571& 0.469\\ \hline 
 $q_\mathrm{MLE}$& $(1.3\pm 0.5) * 10^9$& $49.5 \pm 1.9$& 4.3& 1.97& 1.227& 0.877& 0.675&0.547\\ \hline
    \end{tabular}
    \vspace{0.1cm}
    \caption{Predictive risk of invariant predictive procedures. Only significant digits of the MC estimate are reported if the standard deviation is not explicitly stated.}
    \label{tab:pred-risk}
\end{table}

As we can see, the predictive risk can already be significantly reduced within the plug-in prediction framework by switching from the maximum-likelihood estimate to an unbiased estimator with corrected sample-covariance. However, there is a much bigger improvement from switching to objective Bayesian prediction with Jeffreys prior. Further minor improvements can be achieved by using the independence-Jeffreys prior and then again by using a right-invariant prior instead. The right-invariant prior performs best, consistent with our theoretical optimality result. 

However, the improvements in predictive risk are only significant if the $2n$ degrees of freedom of the observations are close to the $5$ degrees of freedom of parameter space. For example, for $n=3$, $q_\mathrm{R}$ achieves a $\exp(0.39) \approx 1.48$ times higher predictive density on (geometric) average than $q_\mathrm{IJ}$ and $\exp(0.69) \approx 1.99$ times higher than $q_\mathrm{J}$. However, at $n=10$ the improvements are only around $1\%$. While the right-invariant prior is only defined for families of probability distributions with a group structure, and the independence-Jeffreys prior for location-scale models, the Jeffreys prior is defined for any family where the density $p(x \mid \theta)$ is differentiable in $\theta$. This raises the question whether general bounds can be found on how close to optimal the posterior predictive with Jeffreys prior is and when it bounds the predictive risk in the first place.

We visualize the log-likelihood of different predictive procedures in Figure \ref{fig:grid} for a set of three observations (black dots). We can immediately see that the posterior predictive distributions (a-d) are all significantly more heavy-tailed than the plug-in predictive distributions (e) and (f). The posterior predictive of the independence-Jeffreys prior (b) is just a slightly more heavy-tailed version of the Jeffreys prior's posterior predictive (c) but otherwise shares the same elliptical shape. The right-invariant prior posterior predictive on the other hand has a completely different shape and is no longer elliptical. In subfigure (d) we visualize the lack of invariance of the right-invariant prior posterior predictive to reordering the data entries: We swap the $x$ and $y$ coordinates for $\mathbf{x}_1, \dots,\mathbf{x}_n$ and $\mathbf{x}^*$ before querying the predictive procedure. Nevertheless, this modified predictive procedure still has the same predictive risk.

\section{Application to Gaussian Process Regression With Fixed Lengthscale}

We now want to consider a Gaussian Process model with a linear mean and an RBF kernel with fixed lengthscale $\sigma_x > 0$. For given tuples of feature vectors $\mathbf{v}_1, ..., \mathbf{v}_n \in \mathbb{R}^d$ and $\mathbf{w}_1, ..., \mathbf{w}_m \in \mathbb{R}^d$ this gives us the kernel matrix
$$K(\mathbf{v}_{1:n}, \mathbf{w}_{1:m})_{i, j} = \exp\left(-\frac{\left\|\mathbf{v}_i - \mathbf{w}_j\right\|^2}{2 \sigma_x^2}\right)$$
The amplitude $\sigma_y$ is left as a free variable. This results in the likelihood:
$$p(y_{1:n} \mid \mathbf{x}_{1:n}, \sigma_y, \boldsymbol{\beta}) \sim \mathcal{N}([\mathbf{x}_{1:n}]^T \boldsymbol{\beta}, \sigma_y^2 K(\mathbf{x}_{1:n}, \mathbf{x}_{1:n}))$$
where $[\mathbf{x}_{1:n}]$ is the matrix with $\mathbf{x}_{1:n}$ as its columns. Unlike next-sample prediction, the likelihood is not conditionally independent, requiring more attention for the density of $y^*_{1:m}$:
$$p(y^*_{1:m} \mid \mathbf{x}^*_{1:m}, y_{1:n}, \mathbf{x}_{1:n}, \sigma_y, \boldsymbol{\beta})\sim$$
$$\mathcal{N}\left([\mathbf{x}^*_{1:m}]^T\boldsymbol{\beta} + K(\mathbf{x}^*_{1:m}, \mathbf{x}_{1:n}) K(\mathbf{x}_{1:n}, \mathbf{x}_{1:n})^{-1}(y_{1:n} - [\mathbf{x}_{1:n}]^T \boldsymbol{\beta}),\right.$$
$$\left. \sigma_y^2 (K(\mathbf{x}^*_{1:m}, \mathbf{x}^*_{1:m}) - K(\mathbf{x}^*_{1:m}, \mathbf{x}_{1:n}) K(\mathbf{x}_{1:n}, \mathbf{x}_{1:n})^{-1} K(\mathbf{x}_{1:n}, \mathbf{x}^*_{1:m}))\right)$$
We assume that all the vectors we are conditioning on are unique, so $|\{\mathbf{x}_1, ..., \mathbf{x}_n\} \cup \{\mathbf{x}^*_1, ..., \mathbf{x}^*_m\}| = n+m$. This ensures the invertibility and positive-definiteness of the kernel matrices.

\begin{prop}
This likelihood has a group structure with the group $G_{GP}:$ $\mathbf{x} \rightarrow a \mathbf{x} + \mathbf{b}$ with $a \in \mathbb{R}_{>0}$ and $\mathbf{x}, \mathbf{b} \in \mathbb{R}^d$ where the action on $\mathbf{y}_{1:n}$ depends on $[\mathbf{x}_{1:n}]$:
$$\mathbf{y} \rightarrow a \mathbf{y} + [\mathbf{x}_{1:n}]^T\mathbf{b}, \quad (\boldsymbol{\beta}, \sigma_y) \rightarrow (a\boldsymbol{\beta} + \mathbf{b}, a\sigma_y)$$
and the action on $\mathbf{y}_{1:m}^*$ depends on $[\mathbf{x}^*_{1:m}]$, $\mathbf{y}_{1:n}$, and $[\mathbf{x}_{1:n}]$:
$$\mathbf{y}^* \rightarrow a \mathbf{y}^* + [\mathbf{x}^*_{1:m}]^T\mathbf{b} + K(\mathbf{x}^*_{1:m}, \mathbf{x}_{1:n}) K(\mathbf{x}_{1:n}, \mathbf{x}_{1:n})^{-1}((1-a)\mathbf{y} - [\mathbf{x}_{1:n}]^T \mathbf{b})$$
\end{prop}
\begin{proof}
The invariances follow directly from the standard linear transformation rule for multivariate normal random variables: $\mathbf{x} \sim \mathcal{N}(\mu, \Sigma) \Rightarrow A\mathbf{x} + \mathbf{b} \sim \mathcal{N}(A\mu + \mathbf{b}, A \Sigma A^T)$
\end{proof}

From proposition \ref{prop:right-haar} it follows that the right-invariant measure under the action of $G_{GP}$ is $1/\sigma_y$. According to Berger et. al.'s integrated likelihood for this prior (\cite{berger2001objective}, equation 3) the posterior is proper if in addition to the kernel matrix $[\mathbf{x}_{1:n}] K(\mathbf{x}_{1:n}, \mathbf{x}_{1:n})^{-1} [\mathbf{x}_{1:n}]^T$ is positive definite, which is true iff $[\mathbf{x}_{1:n}]$ has full rank. This then guarantees valid predictive procedures. It remains to show the orbit regularity of $\mathcal{Y}^n$. While $G_{GP}$ is an amenable Lie group (as a closed subgroup of the upper-triangular matrix group), its translation-scaling action is not proper. Therefore, Proposition \ref{prop:orb-reg-1} is once again our best starting point for proving orbit-regularity.

\begin{lemma}
Let $\mathbf{x}_1, \dots, \mathbf{x}_n \in \mathbb{R}^d$. If $\rank([\mathbf{x}_{1:n}]) = d$ and $n>d$ then the group actions acts freely on all $y_{1:n} \in \mathcal{Y}^n$ except the linear subspace spanned by the rows of $[\mathbf{x}_{1:n}]$.
\end{lemma}
\begin{proof}
We want to determine the elements $(a, \mathbf{b}) \in G_{GP}$ that stabilize some $\mathbf{y} \in \mathbb{R}^n$. For all $i=1, ..., n$ we would need
$$a y_i + \mathbf{x}_i^T \mathbf{b} = y_i \Leftrightarrow (a-1) y_i + \mathbf{x}_i^T \mathbf{b} = 0$$
For $a=1$ this simplifies to $[\mathbf{x}_{1:n}]^T \mathbf{b} = \mathbf{0}$. Since $\rank([\mathbf{x}_{1:n}]) = d$, it follows from the rank-nullity theorem, that the only solution will be $\mathbf{b} = \mathbf{0}$. For $a\neq 1$, $[\mathbf{x}_{1:n}]^T \mathbf{b} = (1-a)\mathbf{y}_{1:n}$ has a solution if and only if $\mathbf{y}_{1:n}$ lies in the space spanned by the rows of $[\mathbf{x}_{1:n}]$. 
\end{proof}

\begin{prop}
The Euclidean space $\mathbb{R}^n$ with measure $\mathcal{N}([\mathbf{x}_{1:n}]^T \boldsymbol{\beta}, \sigma_y^2 K(\mathbf{x}_{1:n}, \mathbf{x}_{1:n}))$ is orbit-regular under the action of $G_{GP}$ for all $n>d$ distinct feature-tuples $\mathbf{x}_1, \dots, \mathbf{x}_n \in \mathbb{R}^d$ with $\rank([\mathbf{x}_{1:n}]) = d$  and for all $\boldsymbol{\beta} \in \mathbb{R}^d$ and $\sigma_y \in \mathbb{R}_{>0}$.
\end{prop}
\begin{proof}
Firstly, Euclidean space is an oriented smooth manifold on which the multivariate normal measure is given by a nowhere-vanishing volume form (the well-known PDF) as long as the covariance matrix is positive-definite. This is guaranteed by our distinctness assumption. The set $X_0$ is the $d$-dimensional linear subspace spanned by the rows of $[\mathbf{x}_{1:n}]$, so $\{[\mathbf{x}_{1:n}]^T\mathbf{v} \mid \mathbf{v} \in \mathbb{R}^d\}$. As a linear subspace it is closed in $\mathbb{R}^n$, and since $d<n$, it has measure $0$. By our previous lemma, $G_{GP}$ acts freely on $\mathbb{R}^n \setminus X_0$. Since the group action consists of scalar multiplication and addition of vectors that lie in $X_0$, $X_0$ is an orbit of $\mathbb{R}^n$ under the group action, specifically the orbit $G\mathbf{0}$. 

Let us now consider the subspace $W$ that is orthogonal to $X_0$. Every vector in $\mathbb{R}^n \setminus X_0$ can be written as the sum of a vector in $X_0$ and a non-zero vector in $W$. The group action will affect the non-zero $W$ component only by scalar multiplication, so each orbit takes the form $\{a \mathbf{w} \mid \mathbf{w} \in W, a \in \mathbb{R}_{>0}\}$. Therefore, the quotient space has the smooth structure of an $(n-d-1)$-sphere. The quotient map is described by orthogonal projection onto $W$ followed by the projection $\mathbf{w} \rightarrow \mathbf{w} / \|\mathbf{w}\|$ (which is a projection in polar coordinates). Since projections are submersions and the composition of submersions is a submersion, the quotient map is a submersion.
\end{proof}

Hence, we have another corollary of the main result:
\begin{coro} \label{coro:gp}
An optimal predictive procedure for a Gaussian process with linear mean $\mathbf{x}^T \boldsymbol{\beta}$ and RBF kernel with fixed lengthscale and scale $\sigma_y$ in the sense of (\ref{eq:maximin}) when having made $n > d$ observations $y_{1:n}$ at the distinct locations $\mathbf{x}_1, \dots, \mathbf{x}_n \in \mathbb{R}^d$ with $\rank([\mathbf{x}_{1:n}]) = d$ and $y_{1:n}$ not in the row space of $[\mathbf{x}_{1:n}]$, and wanting to make a prediction at the locations $\mathbf{x}^*_1, \dots, \mathbf{x}^*_m \in \mathbb{R}^d$ is given by the Bayesian posterior predictive with improper prior
$$\mu_R(\boldsymbol{\beta}, \sigma_y) \propto 1/\sigma_y d\boldsymbol{\beta}d\sigma_y$$
Additionally, this predictive procedure is invariant under $G_{GP}$.
\end{coro}
For Gaussian Processes, the right-invariant prior is unique (up to a scalar factor).

\begin{prop} \label{prop:gp-pred}
Under the same assumptions of Corllary \ref{coro:gp} and additionally assuming that $\mathbf{x}^*_1, \dots, \mathbf{x}^*_m$ are distinct from each other and $\mathbf{x}_1, ..., \mathbf{x}_n$, the optimal predictive procedure can be identified as the multivariate $t$-distribution
$$t_{n-d}\left(-A_{pp}^{-1} A_{po} \mathbf{y}_{1:n}, \frac{\mathbf{y}_{1:n}^T (A_{oo} - A_{op} A_{pp}^{-1} A_{po}) \mathbf{y}_{1:n}}{n - d} A_{pp}^{-1}\right)$$
where $A := K^{-1/2} (I - Z (Z^T Z)^{-1} Z^T) K^{-1/2}$ with $K := K([\mathbf{x}_{1:n}, \mathbf{x}^*_{1:m}], [\mathbf{x}_{1:n}, \mathbf{x}^*_{1:m}])$ and $Z := K^{-1/2} [\mathbf{x}_{1:n}, \mathbf{x}^*_{1:m}]^T$ is decomposed into submatrices as $A = \begin{pmatrix}
    A_{oo} & A_{op} \\
    A_{po} & A_{pp}
\end{pmatrix}$ where $A_{oo} \in \mathbb{R}^{n \times n}$, $A_{op} \in \mathbb{R}^{n \times m}$, $A_{po} \in \mathbb{R}^{m \times n}$, and $A_{pp} \in \mathbb{R}^{m \times m}$.  
\end{prop}
\begin{proof}
See Appendix \ref{sec:gp-pred-proof}
\end{proof}

Since we have a multivariate normal likelihood where the mean and covariance are separately parametrized $\mathcal{N}(X\boldsymbol{\beta}, \sigma_y^2 K)$, we can determine the Fisher information matrix using the formula from (\cite{mardia1984maximum}, equation 2.5). This results in the Fisher information matrix $\mathcal{I}(\boldsymbol{\beta}, \sigma_y) = \operatorname{diag}\left(\frac{1}{\sigma_y^2} X^T K^{-1} X, \frac{2n}{\sigma_y^2}\right)$.
Hence, Jeffrey's prior is given by $\sqrt{\det(\mathcal{I}(\boldsymbol{\beta}, \sigma_y))} \propto \sigma_y^{-(p+1)}$. Analogously to the right-invariant prior, this leads to another multivariate $t$ predictive distribution:
$$t_n\left(-A_{pp}^{-1} A_{po} \mathbf{y}_{1:n}, \frac{\mathbf{y}_{1:n}^T (A_{oo} - A_{op} A_{pp}^{-1} A_{po}) \mathbf{y}_{1:n}}{n} A_{pp}^{-1}\right)$$
The independence-Jeffreys prior agrees with the right-invariant prior in this case.

For maximum-likelihood estimation, the estimate of $\boldsymbol{\beta}$ is just the standard GLS estimator 
$$\boldsymbol{\hat{\beta}} = \left(X^T K^{-1} X\right)^{-1} X^T K^{-1} \mathbf{y}, \qquad \hat{\sigma_y} = \sqrt{(\mathbf{y} - X \boldsymbol{\hat{\beta}})^T K^{-1} (\mathbf{y} - X \boldsymbol{\hat{\beta}}) / n}$$ 
For an unbiased estimate, we divide by $n-p$ instead.

\subsection{Numerical Results}
For our test we use a standard spatial setting where the features are an $x$ and $y$ coordinate as well as a constant feature $1$, to allow for any affine function of the coordinates as the mean. The $x_1$ and $x_2$ coordinates for observation and prediction are arbitrarily chosen from $(0, 1) \times (0, 1)$ and visualized in Figure \ref{fig:observed-points-gp}. The lengthscale of the RBF kernel is set to 1. We always predict at the same point and add a new observation location but keep the old ones as $n$ increases. We summarized the results of our Monte Carlo simulation in Table \ref{tab:pred-risk-gp}.


\begin{figure}[t]
\centering

\begin{minipage}[t]{0.4\textwidth}
    \centering
    \includegraphics[width=\linewidth]{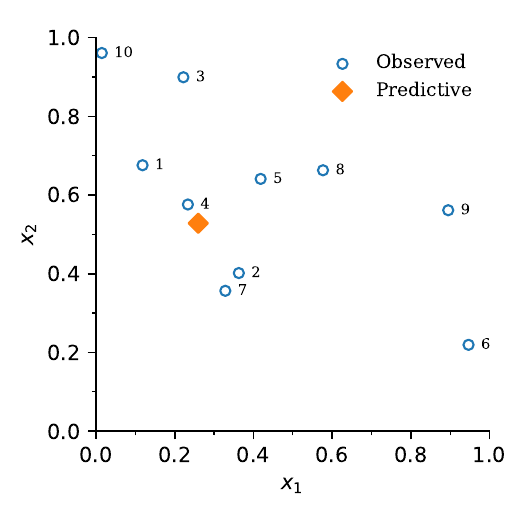}
    \captionof{figure}{Observation and prediction points}
    \label{fig:observed-points-gp}
\end{minipage}\hfill
\begin{minipage}[t]{0.5\textwidth} \vspace{-4cm}
    \centering
    \small
    \setlength{\tabcolsep}{4pt}
    \renewcommand{\arraystretch}{1.1}

    \begin{tabular}{@{}ccccc@{}}
    \toprule
    1 & 2 & 3 & 4 & 5 \\
    \midrule
    1.598 & 3.702 & 3.714 & 5.830 & 6.285 \\
    \bottomrule
    \end{tabular}

    \vspace{4pt}

    \begin{tabular}{@{}ccccc@{}}
    \toprule
    6 & 7 & 8 & 9 & 10 \\
    \midrule
    6.338 & 6.779 & 7.401 & 7.438 & 7.753 \\
    \bottomrule
    \end{tabular}

    \captionof{table}{Improvement in predictive performance knowing true parameters after observing at $n$ points relative to no observations}
    \label{tab:true-improvement-gp}
\end{minipage}

\end{figure}

\begin{table}
    \centering
    \begin{tabular}{|c|c|c|c|c|c|c|c|} \hline 
         n&  4&  5&  6&  7&  8&  9& 10\\ \hline 
         $q_\mathrm{R} = q_{\mathrm{IJ}}$&  0.594&  0.461&  0.3297&  0.459&  0.2483&  0.2176& 0.1320\\ \hline 
 $q_\mathrm{J}$& 1.815& 0.852& 0.522& 0.573& 0.324& 0.272&0.1725\\ \hline 
         $q_\mathrm{unb}$&  MC var. too high&  $36 \pm 22$&  1.38&  1.30&  0.510&  0.390& 0.213\\ \hline 
 $q_\mathrm{MLE}$& MC var. too high& $91 \pm 55$& 3.09& 2.48& 0.945& 0.676&0.373\\ \hline
    \end{tabular}
    \vspace{0.1cm}
    \caption{Predictive risk of invariant predictive procedures for GP (significant digits)}
    \label{tab:pred-risk-gp}
\end{table}

As expected, the right-invariant prior performs best. Due to the observations being dependent, it can happen that the log-likelihood we can assign knowing the true parameters improves more from additional observations than the predictive procedures. We observe this for the step from $n=6$ to $n=7$. Although the absolute performance of $q_R$ improves from the additional observation, the relative performance degrades. The relative performance of the plug-in prediction methods consistently improves but is always significantly worse than the objective Bayes methods, similar to what we have observed for the bivariate normal distribution. Plug-in prediction using the unbiased estimator outperforms the MLE again.

\begin{prop}
For our Gaussian Process model, the improvement in predictive performance at a given point from an additional observation knowing the true parameters is equal for any value of the true parameters.
\end{prop}

This proposition follows from the entropy of the conditional normal distributions and allows us to directly compare the predictive performance for any combination of an invariant predictive procedure and a number of observations $n$, even without knowledge of the true parameters. In Table \ref{tab:true-improvement-gp} we evaluate how the baseline of the predictive risk, the predictive performance knowing the true parameters, has improved after $n$ observations. This allows us to directly compare the performance of any combination of number of observations and predictive procedure.

\newpage
\bibliographystyle{unsrt}  
\bibliography{main}

\newpage
\appendix
\section{Liang's Subprobability Measure Assumption} \label{sec:liang-assumptions}
\begin{prop}\label{prop:measurable-action}
Under our assumption of a measurable group action $\alpha_{\mathcal{Y}}: G \times \mathcal{Y} \rightarrow \mathcal{Y}$ such that
$$\forall g \in G, \theta \in \Theta, A \in B(\mathcal{Y}): p(A \mid \theta) = p(gA \mid g\theta)$$
we can always define a group action of $G$ on the space $S$ of \emph{all} subprobability measures on $\mathcal{Y}^m$ such that $\forall g \in G, \theta \in \Theta, Q \in S: D_{KL}(p(\cdot \mid \theta) || Q) = D_{KL}(p(\cdot \mid g\theta) || gQ)$ and for any bounded measurable function $f$ on $\mathcal{Y}^m$ we have $\forall g \in G, Q \in S, A \in \sigma(\mathcal{Y}^*): \int f(A) d\mu(gQ) = \int f(gA) d\mu(Q)$
\end{prop}
\begin{proof}
Under these assumptions, we can define a group action on $S$ as the pushforward measure under the group action of $G$ on $\mathcal{Y}^*$. Since $\forall g \in G, \theta \in \Theta, A \in B(\mathcal{Y}^m): p(A \mid \theta) = p(gA \mid g\theta)$ we can also see $p(\cdot \mid g\theta)$ as the pushforward measure of the group action of $G$ on $\mathcal{Y}^*$. Since the KL-divergence is invariant to a change of variables, we have $\forall g \in G, \theta \in \Theta, Q \in S: D_{KL}(p(\cdot \mid \theta) || Q) = D_{KL}(p(\cdot \mid g\theta) || gQ)$. The change of variables theorem for pushforward measures also implies that $\forall g \in G, Q \in S, A \in \sigma(\mathcal{Y}^*): \int f(A) d\mu(gQ) = \int f(gA) d\mu(Q)$ if both sides are integrable.
\end{proof}

\section{Regular Value Proof} \label{sec:regular-value-proof}
We want to show that $(\mathbf{\bar{x}} = \mathbf{0}, \bar{\Sigma} = I)$ is a regular value of the map from $n>d$ observations $\mathbf{x}_1, ..., \mathbf{x}_n \in \mathbb{R}^d$ to their sample mean $\mathbf{\bar{x}}$ and sample covariance $\bar{\Sigma}$. We will start by showing the statement for $n = d+1$ points. $d+1$ points $\mathbf{x}_1, ..., \mathbf{x}_{d+1}$ in general position are also \emph{affinely independent} which can be characterized by the nonexistence of factors $a_1, ..., a_{d+1}$ that are not all zero such that 
$$\sum_{i=1}^{d+1} a_i \mathbf{x}_i = \mathbf{0} \text{ and } \sum_{i=1}^{d+1} a_i = 0$$
We first want to show the following lemma:
\begin{lemma}
If we have $d+1$ points $\mathbf{x}_1, ..., \mathbf{x}_{d+1} \in \mathbb{R}^d$ in general position, we can for any vector $\mathbf{y} \in \mathbb{R}^d$ find $a_1, ..., a_{d+1} \in \mathbb{R}$ such that
$$\sum_{i=1}^{d+1} a_i \mathbf{x}_i = \mathbf{y} \text{ and } \sum_{i=1}^{d+1} a_i = 0$$
\end{lemma}
\begin{proof}
For fixed $\mathbf{x}_1, ..., \mathbf{x}_{d+1} \in \mathbb{R}^d$ we consider the linear transformation $T: \mathbb{R}^{d+1} \rightarrow \mathbb{R}^d$ defined as:
$$T(a_1, ..., a_{d+1}) = \sum_{i=1}^{d+1} a_i \mathbf{x}_i$$
Since $\mathbf{x}_1, ..., \mathbf{x}_{d+1}$ are in general position, any subset of $d$ vectors forms a basis of $\mathbb{R}^d$, implying that $\rank(T) = d$. By the rank-nullity theorem, we can conclude that $\dim(\ker(T))=1$. Hence, there exists a unique up to a scalar factor nonzero vector $\boldsymbol \lambda \in \mathbb{R}^{d+1}$ such that
$$\sum_{i=1}^{d+1} \lambda_i \mathbf{x}_i = \mathbf{0}$$
Since $\mathbf{x}_1, ..., \mathbf{x}_{d+1} \in \mathbb{R}^d$ are in general position, they are affinely independent and we can conclude that $\sum_{i=i}^{d+1} \lambda_i \neq 0$.

Let us now define the linear subspace $S$ of $\mathbb{R}^{d+1}$:
$$S = \left\{(a_1, ..., a_{d+1}) \in \mathbb{R}^{d+1}: \sum_{i=1}^{d+1} a_i = 0\right\}$$
Since there is one linear constraint, the dimension of $S$ is $n$.

Now we see that $\ker(T_{|S})  = \ker(T) \cup S = \{\mathbf{0}\}$. Hence $T_{|S}$ is injective and due to the dimension equality also surjective, proving the lemma.
\end{proof}
\begin{coro}
If we have $d+1$ points $\mathbf{x}_1, ..., \mathbf{x}_{d+1} \in \mathbb{R}^d$ in general position with $\frac{1}{d+1}\sum_{i=1}^{d+1} \mathbf{x}_i = \mathbf{0}$, we can for any vector $\mathbf{y} \in \mathbb{R}^d$ and $c \in \mathbb{R}$ find $a_1, ..., a_{d+1} \in \mathbb{R}$ such that
$$\sum_{i=1}^{d+1} a_i \mathbf{x}_i = \mathbf{y} \text{ and } \sum_{i=1}^{d+1} a_i = c$$
\end{coro}
\begin{proof}
According to the previous lemma we can find $b_1, ..., b_{d+1}$ such that
$$\sum_{i=1}^{d+1} b_i \mathbf{x}_i = \mathbf{y} \text{ and } \sum_{i=1}^{d+1} b_i = 0$$
Now we define $a_i = b_i + \frac{c}{d+1}$.
\end{proof}

\begin{prop} \label{prop:whitened-manifold}
If we have a set of $n > d$ points $\mathbf{x}_1, ..., \mathbf{x}_n \in \mathbb{R}^d$ in general position, and the function:
$$\mathbf{\bar{x}}=\frac{1}{n}\sum_{i=1}^{n}\mathbf{x}_i, \quad \bar{\Sigma} = {1 \over {n-1}}\sum_{i=1}^n (\mathbf{x}_i-\mathbf{\bar{x}}) (\mathbf{x}_i-\mathbf{\bar{x}})^T$$
Then $(\mathbf{\bar{x}} = \mathbf{0}, \bar{\Sigma} = I)$ is a regular value.
\end{prop}
\begin{proof}
$(\mathbf{\bar{x}} = \mathbf{0}, \bar{\Sigma} = I)$ is a regular value iff every point in its preimage has a surjective differential. For $\mathbf{\bar{x}} = \mathbf{0}$ the differential is given by
$$d\mathbf{\bar{x}} = \frac{1}{n} \sum_{i=1}^n d\mathbf{x}_i, \quad d\bar{\Sigma} = \frac{1}{n-1} \sum_{i=1}^n \mathbf{x}_i d\mathbf{x}_i^T + d\mathbf{x}_i \mathbf{x}_i^T$$
Since we have $n > d$ points in general position, $\hat{\Sigma}$ is symmetric positive definite and the tangent space of symmetric positive definite matrices are the symmetric matrices $\mathrm{Sym}_n$. For any value of $(d\mathbf{\bar{x}} \in \mathbb{R}^d, d\bar{\Sigma} \in \mathrm{Sym_n})$ we now need to find $d\mathbf{x}_1, ..., d\mathbf{x}_{d+1}$ such that the equations hold. The equations simplify to:
$$\sum_{i=1}^n d\mathbf{x}_i = n d\mathbf{\bar{x}}, \quad \sum_{i=1}^n \mathbf{x}_i d\mathbf{x}_i^T = 2(n-1) d\bar{\Sigma}$$
If we look at the matrix $2(n-1) d\bar{\Sigma}$ column-wise, we see that we have a linear combination of $n$ basis vectors $\mathbf{x}_1, ..., \mathbf{x}_n$ that need to equal that column vector, and an additional constraint on the sum of the factors. Due to the previous corollary, for $n=d+1$ we can always choose $d\mathbf{x}_1, ..., d\mathbf{x}_{d+1}$ such that the equations hold.  For $n > d+1$ we can set $d\mathbf{x}_{d+2}, ..., d\mathbf{x}_{n}$ to $\mathbf{0}$. This proves the theorem.
\end{proof}

\section{Posterior Propriety} \label{sec:posterior-propriety}
\begin{thm}
The posterior of the $d$-dimensional multivariate normal distribution after $n>d$ observations with a right-invariant prior is almost surely proper.
\end{thm}
\begin{proof}
Consider the posterior density:
$$p(U, \boldsymbol \mu) \propto \det (U)^{-n} \, \exp \left( -\frac{1}{2} \sum_{i=1}^n (\mathbf{y}_i - \boldsymbol\mu)^\mathrm{T} \left(U U^T\right)^{-1}(\mathbf{y}_i - \boldsymbol\mu) \right) \prod_{i=1}^d (U_{i,i})^{-i}$$
Let us start by integrating out $\mu$ and let us for now write $\Sigma = UU^T$ and ignore terms constant in $\mu$:
$$\int \exp \left( -\frac{1}{2} \sum_{i=1}^n (\mathbf{y}_i - \boldsymbol\mu)^\mathrm{T} \Sigma^{-1}(\mathbf{y}_i - \boldsymbol\mu) \right) d\boldsymbol{\mu}$$
Let us introduce $\mathbf{\bar{y}} = \frac{1}{n} \sum_{i=1}^n \mathbf{y}_i$ and replace $\mathbf{y}_i - \boldsymbol \mu = (\mathbf{y}_i - \mathbf{\bar{y}}) + (\mathbf{\bar{y}} - \boldsymbol \mu)$. Expanding, this gives
$$=\int \exp \left( -\frac{1}{2} \left(\sum_{i=1}^n (\mathbf{y}_i - \bar{\mathbf{y}})^\mathrm{T} \mathbf{\Sigma}^{-1}(\mathbf{y}_i - \bar{\mathbf{y}}) + n (\boldsymbol\mu - \bar{\mathbf{y}})^\mathrm{T} \mathbf{\Sigma}^{-1}(\boldsymbol\mu - \bar{\mathbf{y}})\right) \right) d\boldsymbol{\mu}$$
$$=\exp\left(-\frac{1}{2} \sum_{i=1}^n (\mathbf{y}_i - \bar{\mathbf{y}})^\mathrm{T} \mathbf{\Sigma}^{-1}(\mathbf{y}_i - \bar{\mathbf{y}})\right) \int \exp\left(-\frac{1}{2} (\boldsymbol\mu - \bar{\mathbf{y}})^\mathrm{T} (\mathbf{\Sigma}/n)^{-1}(\boldsymbol\mu - \bar{\mathbf{y}})\right)d\boldsymbol{\mu}$$
$$=\exp\left(-\frac{1}{2} \sum_{i=1}^n (\mathbf{y}_i - \bar{\mathbf{y}})^\mathrm{T} \mathbf{\Sigma}^{-1}(\mathbf{y}_i - \bar{\mathbf{y}})\right) (2\pi)^{n/2} \det(\Sigma)^{1/2}$$
Taking into account $\det (U)^{-n} $ and $\det(\Sigma)^{1/2}=\det(U)$, we now have:
$$\int p(U, \boldsymbol \mu) d\boldsymbol{\mu} \propto \exp\left(-\frac{1}{2} \sum_{i=1}^n (\mathbf{y}_i - \bar{\mathbf{y}})^\mathrm{T} \left(U U^T\right)^{-1}(\mathbf{y}_i - \bar{\mathbf{y}})\right) \prod_{i=1}^d (U_{i,i})^{-i-n+1}$$
We perform a change of variables $H=U^{-1}$. The Jacobian determinant (see e.g. \cite{Mathai2008}, Note 11.2.2) is $dU = |H|^{-(d+1)} dH$.
$$=\exp\left(-\frac{1}{2} \sum_{i=1}^n (\mathbf{y}_i - \bar{\mathbf{y}})^\mathrm{T} H^T H (\mathbf{y}_i - \bar{\mathbf{y}})\right) \prod_{i=1}^d (H_{i,i})^{i + n - d - 2}$$
$$=\exp\left(-\frac{1}{2} \mathbf{Tr}\left(H \left(\sum_{i=1}^n(\mathbf{y}_i - \bar{\mathbf{y}}) (\mathbf{y}_i - \bar{\mathbf{y}})^T \right) H^T\right)\right)\prod_{i=1}^d (H_{i,i})^{i + n - d - 2}$$
If $n>d$, the scaled sample covariance matrix $S = \sum_{i=1}^n(\mathbf{y}_i - \bar{\mathbf{y}}) (\mathbf{y}_i - \bar{\mathbf{y}})^T $ is positive definite almost surely. Now $HSH^T \succeq \lambda_{\textrm{min}} HH^T$ where $\lambda_{\textrm{min}}$ is the smallest eigenvalue of $S$. Hence, $\mathbf{Tr}(HSH^T) \geq \lambda_{\mathrm{min}} \|H\|_F^2$. Now we have
$$\leq \exp\left(-\frac{\lambda_{\textrm{min}}}{2} \|H\|_F^2\right)\prod_{i=1}^d (H_{i,i})^{i + n - d - 2}$$
Now, instead of integrating over $H$ element-wise, we preform another change of variables and integrate over the radius $r = \|H\|_F$ and angle $\Theta = H / \|H\|_F$ separately, $dH = r^{\frac{d(d+1)}{2} - 1} drd\Theta$.
$$= \exp\left(-\frac{\lambda_{\textrm{min}}}{2} r^2\right) r^{d(n-1)-1} \prod_{i=1}^d \Theta_{i,i}^{i + n - d - 2}$$
Since $d \geq 1$ and $n \geq 2$, we have $d(n-1)-1 \geq 0 > -1$. Hence,  we have a gamma-integral:
$$\int \exp\left(-\frac{\lambda_{\textrm{min}}}{2} r^2\right) r^{d(n-1)-1} dr = 2^{\frac{1}{2} d (n-1)-1} \lambda_{\textrm{min}}^{-\frac{1}{2} d (n-1)} \Gamma \left(\frac{1}{2} d (n-1)\right) =: c$$
Now we have the integral $c \int \prod_{i=1}^d \Theta_{i,i}^{i + n - d - 2} d\Theta$. Since $i+n-d-2 \geq 0$ the function is bounded on the compact closure of $\Theta$, where the diagonal elements may also be $0$. Hence, the integral is finite.
\end{proof}

\section{Closed-Form Predictive Procedure for Bivariate Normal Distribution} \label{sec:bivn-predictive}
We are considering the bivariate normal distribution parametrized as 
$$(u_i, v_i) \sim \mathcal{N}\left(\begin{pmatrix}
\mu_u \\
\mu_v
\end{pmatrix},
\begin{pmatrix}
a & b \\
0 & c
\end{pmatrix}
\begin{pmatrix}
a & b \\
0 & c
\end{pmatrix}^T\right)$$
The probability density function is
$$\frac{\exp \left(-\frac{\left(a^2+b^2\right) (\mu_v-v_i)^2/c^2-2 b (\mu_u-u_i) (\mu_v-v_i)/c+(\mu_u-u_i)^2}{2 a^2}\right)}{2 \pi  a c}$$
Now we of course have not only one observation but $n$ observations $\mathbf{x}_1 = (u_1, v_1)^T, ..., \mathbf{x}_n = (u_n, v_n)^T$. Nevertheless, the exponential family nature of the bivariate normal distribution allows us to maintain a finite number of sufficient statistics: $\mathbf{u} \cdot\mathbf{1}$, $\mathbf{v} \cdot\mathbf{1}$, $\mathbf{u} \cdot \mathbf{u}$, $\mathbf{v} \cdot\mathbf{v}$, and $\mathbf{u} \cdot\mathbf{v}$: 
$$\exp \left(\left(-\frac{\mu_v^2 \left(a^2+b^2\right)}{2 a^2 c^2}+\frac{b \mu_u \mu_v}{a^2 c}-\frac{\mu_u^2}{2 a^2}\right) n +\left(\frac{\mu_v \left(a^2+b^2\right)}{a^2 c^2}-\frac{b \mu_u}{a^2 c}\right)\mathbf{v} \cdot\mathbf{1}-\frac{\left(a^2+b^2\right)}{2 a^2 c^2} \mathbf{v} \cdot \mathbf{v}\right.$$
$$\left.+\left(\frac{\mu_u}{a^2}-\frac{b \mu_v}{a^2 c}\right)\mathbf{u} \cdot\mathbf{1}+\frac{b}{a^2 c}\mathbf{u} \cdot \mathbf{v}-\frac{1}{2 a^2}\mathbf{u} \cdot \mathbf{u}\right)/(2 \pi  a c)^n$$
To obtain a predictive density we will now proceed by integrating out the parameters $\mu_u$, $\mu_v$, $a$, $b$, and $c$ based on the right-invariant prior $a^{-1} c ^{-2}$. We are using Mathematica. Integrating out $\mu_u$ and $\mu_v$, we have the integrated (over the prior) likelihood
$$(2 \pi )^{1-n} (a c)^{-n} \frac{\exp \left(-\frac{\left(a^2+b^2\right) \left(n \mathbf{v}\cdot\mathbf{v} - (\mathbf{v}\cdot\mathbf{1})^2\right)+2 b c ((\mathbf{u}\cdot\mathbf{1}) (\mathbf{v}\cdot\mathbf{1}) - n \mathbf{u}\cdot\mathbf{v})+c^2 \left(n \mathbf{u}\cdot\mathbf{u} - (\mathbf{u}\cdot\mathbf{1})^2\right)}{2 a^2 c^2 n}\right)}{c n}$$
After further integrating out $b$, $a$, and $c$ we arrive at an integrated likelihood proportional to
$$\frac{\left(n (\mathbf{u}\cdot\mathbf{u}) (\mathbf{v}\cdot\mathbf{v})-n (\mathbf{u}\cdot\mathbf{v})^2-(\mathbf{u}\cdot\mathbf{1})^2 \mathbf{v}\cdot\mathbf{v}+2 (\mathbf{u}\cdot\mathbf{1}) (\mathbf{u}\cdot\mathbf{v}) (\mathbf{v}\cdot\mathbf{1})-\mathbf{u}\cdot\mathbf{u} (\mathbf{v}\cdot\mathbf{1})^2\right)^{-\frac{n-2}{2}}}{n \mathbf{v}\cdot\mathbf{v} -(\mathbf{v}\cdot\mathbf{1})^2}$$
Note that the numerator is a power of the determinant of the Gram matrix of $\mathbf{u}$, $\mathbf{v}$, and $\mathbf{1}_n$ (the all one vector in $\mathbb{R}^n$) and the denominator is the determinant of the Gram matrix of $\mathbf{v}$ and $\mathbf{1}_n$.
$$\frac{\det(G(\mathbf{u}, \mathbf{v}, \mathbf{1}_n))^{-\frac{n-2}{2}}}{\det(G(\mathbf{v}, \mathbf{1}_n))}$$
To get a next-sample predictive distribution over the next observation $\mathbf{x}^* = (u^*, v^*)^T$ given the $n$ observations $\mathbf{u} = (x_1, ..., x_n)^T$, $\mathbf{v} = (y_1, ..., y_n)^T$ we can now just update the sufficient statistics by $\mathbf{x}^*$:
\begin{equation}
p(u^*, v^* \mid \mathbf{u}, \mathbf{v}) \propto \frac{\det\left(G\left(\begin{pmatrix}\mathbf{u} \\ u^* \end{pmatrix}, \begin{pmatrix}\mathbf{v} \\ v^* \end{pmatrix}, \mathbf{1}_{n+1}\right)\right)^{-\frac{n-1}{2}}}{\det\left(G\left(\begin{pmatrix}\mathbf{v} \\ v^* \end{pmatrix}, \mathbf{1}_{n+1}\right)\right)}
\label{eq:integrated-likelihood}
\end{equation}
The main challenge is now determining the normalizing constant. We will first integrate out $u^*$ since the unnormalized density is the power of a quadratic function in $u^*$, allowing us to use the integral formula:
$$\int_{-\infty}^{\infty} (a x^2 + 2b x + c)^{p}\,dx 
\;=\;
a^{-p-1}\,\left(ac - b^2\right)^{\,p+\tfrac12}\,\sqrt{\pi}\,\frac{\Gamma\!\left(-p - \tfrac12\right)}{\Gamma\!\left(-p\right)}
$$
The formula holds and is positive for any real $a, b, c, p$ such that $p < -\frac{1}{2}$ and the quadratic function only takes positive values, so $a>0$ and $ac > b^2$.


After integrating out $u^*$ from (\ref{eq:integrated-likelihood}) we then have:
$$\sqrt{\pi}\,\frac{\Gamma\!\left(\frac{n-2}{2}\right)}{\Gamma\!\left(\frac{n-1}{2}\right)}\det(G(\mathbf{v}, \mathbf{1}_n))^{\frac{n-3}{2}} \det(G(\mathbf{u}, \mathbf{v}, \mathbf{1}_n))^{-\frac{n-2}{2}} \det\left(G\left(\begin{pmatrix}\mathbf{v} \\ v^* \end{pmatrix}, \mathbf{1}_{n+1}\right)\right)^{-\frac{n}{2}} $$


Applying the integration rule again to $v^*$, this leads to the overall normalizing constant:
$$\frac{2\pi}{n-2} \frac{n^{(n-2)/2}}{(n+1)^{(n-1)/2}} \frac{\det(G(\mathbf{v}, \mathbf{1}_n))^{-\frac{n-1}{2}}}{\sqrt{\det(G(\mathbf{u}, \mathbf{v}, \mathbf{1}_n))}}$$
Hence, the predictive procedure is given by:
$$
q(u^*, v^* \mid \mathbf{u}, \mathbf{v}) = \frac{n-2}{2\pi} \frac{(n+1)^{(n-1)/2}}{n^{(n-2)/2}} \frac{\det(G( \mathbf{v}, \mathbf{1}_n))}{\det(G(\mathbf{u}, \mathbf{v}, \mathbf{1}_n))^{-\frac{n-2}{2}}}
\frac{\det\left(G\left(\begin{pmatrix}\mathbf{u} \\ u^* \end{pmatrix}, \begin{pmatrix}\mathbf{v} \\ v^* \end{pmatrix}, \mathbf{1}_{n+1}\right)\right)^{-\frac{n-1}{2}}}{\det\left(G\left(\begin{pmatrix}\mathbf{v} \\ v^* \end{pmatrix}, \mathbf{1}_{n+1}\right)\right)}
$$

\section{Closed-Form Predictive Procedure for Gaussian Processes}\label{sec:gp-pred-proof}
Proof of Proposition \ref{prop:gp-pred}: We define $X := [\mathbf{x}_{1:n}, \mathbf{x}^*_{1:m}]^T$,  $\mathbf{y} := (y_1, ..., y_n, y^*_1, ..., y^*_m)^T$, $k := n + m$, and $K := K([\mathbf{x}_{1:n}, \mathbf{x}^*_{1:m}], [\mathbf{x}_{1:n}, \mathbf{x}^*_{1:m}])$. We will first integrate out $\boldsymbol{\beta}$:
$$p(\mathbf{y} \mid \boldsymbol{\beta}, \sigma_y) =\int_{\mathbb{R}^d} (2 \pi \sigma_y^2)^{-k/2} \det(K)^{-1/2} \exp\left(-\frac{1}{2 \sigma_y^2} (\mathbf{y} - X\boldsymbol{\beta})^T K^{-1} (\mathbf{y} - X \boldsymbol{\beta})\right) d\boldsymbol{\beta}$$
Let $M = X^T K^{-1} X$ and $\hat{\boldsymbol{\beta}} = M^{-1} X^T K^{-1}\mathbf{y}$. Then:
$$=\int_{\mathbb{R}^d} (2 \pi \sigma_y^2)^{-k/2} \det(K)^{-1/2} \exp\left(-\frac{1}{2 \sigma_y} \left((\boldsymbol{\beta} - \hat{\boldsymbol{\beta}})^T M (\boldsymbol{\beta} - \hat{\boldsymbol{\beta}}) + \mathbf{y}^T K^{-1} \mathbf{y} - \hat{\boldsymbol{\beta}}^T M \hat{\boldsymbol{\beta}}\right) \right) d\boldsymbol{\beta}$$
$$=\int_{\mathbb{R}^d} (2 \pi \sigma_y^2)^{(d-k)/2} \det(K)^{-1/2} \det(M)^{-1/2} \exp\left(-\frac{1}{2 \sigma_y} \left( \mathbf{y}^T K^{-1} \mathbf{y} - \hat{\boldsymbol{\beta}}^T M \hat{\boldsymbol{\beta}} \right) \right) d\boldsymbol{\beta}$$
$$= (2\pi \sigma_y^2)^{(d-k)/2} \det(K)^{-1/2} \det(X^T K^{-1} X)^{-1/2} \exp\left(-\frac{1}{2\sigma_y^2} \mathbf{y}^T \left[ K^{-1} - K^{-1} X(X^T K^{-1} X)^{-1} X^T K^{-1}\right] \mathbf{y}\right)$$
Let us write $A = K^{-1} - K^{-1} X(X^T K^{-1} X)^{-1} X^T K^{-1}$. Defining $Z := K^{-1/2} X$, $A$ can be seen as an orthogonal projection onto the orthogonal complement of $\mathrm{col}(Z)$, in a different basis given by $K^{-1/2}$: We have $A = K^{-1/2} (I - Z (Z^T Z)^{-1} Z^T) K^{-1/2}$. Sylvester's law of inertia confirms that $A$ is positive semi-definite. Further, we now know that $\ker(I - Z (Z^T Z)^{-1} Z^T) = \mathrm{col}(Z)$ and hence $\ker(A) = \mathrm{col}(X)$. 

Since we have assumed that $\mathbf{y}_{1:n}$ does not lie in the row space of $[\mathbf{x}_{1:n}]$, $\mathbf{y}$ does not lie in $\mathrm{col}(X)$, $\mathbf{y}^T A \mathbf{y} > 0$. Now we can integrate out $\sigma_y$ as well:
$$\int_{\mathbb{R}_{>0}} (2\pi \sigma_y^2)^{(d-k)/2} |K|^{-1/2} |X^T K^{-1} X|^{-1/2} \exp\left(-\frac{\mathbf{y}^T A \mathbf{y}}{2\sigma_y^2}\right) \frac{1}{\sigma_y} d\sigma_y$$
$$=\Gamma\left(\frac{k-d}{2}\right) \frac{\pi^{(d-k)/2}}{2} |K|^{-1/2} |X^T K^{-1} X|^{-1/2}  \left(\mathbf{y}^T A \mathbf{y}\right)^{(d-k)/2} \propto \left(\mathbf{y}^T A \mathbf{y}\right)^{(d-k)/2}$$

Remembering $\mathbf{y} := (y_1, ..., y_n, y^*_1, ..., y^*_m)^T$, we can then see that $y^*_1, ..., y^*_m$ follow a multivariate $t$-distribution. Let us now determine its parameters. We start by decomposing $A$ into submatrices as $A = \begin{pmatrix}
    A_{oo} & A_{op} \\
    A_{po} & A_{pp}
\end{pmatrix}$ where $A_{oo} \in \mathbb{R}^{n \times n}, A_{op} \in \mathbb{R}^{n \times m}, A_{po} \in \mathbb{R}^{m \times n}, A_{pp} \in \mathbb{R}^{m \times m}$. Then
$$= \left({\mathbf{y}^*_{1:m}}^T A_{pp} {\mathbf{y}^*_{1:m}} + 2 \mathbf{y}_{1:n}^T A_{op} {\mathbf{y}^*_{1:m}} + \mathbf{y}_{1:n}^T A_{oo} \mathbf{y}_{1:n}\right)^{(d-k)/2}$$
As a submatrix of a positive semi-definite matrix, $A_{pp}$ is positive semi-definite. For positive semi-definite matrices $M$, we have $\forall \mathbf{v}: \mathbf{v}^T M \mathbf{v} = 0 \Leftrightarrow M \mathbf{v} = \mathbf{0}$. Further, since $\left[\mathbf{x}_{1:n}\right]$ has full rank, the only vector in $\ker(A)$ that has the first $n$ entries 0 is the 0-vector. Hence, $A_{pp}$ is positive definite. The positive minimum of the quadratic function in ${\mathbf{y}^*_{1:m}}$ is given by $\mathbf{y}_{1:n}^T A_{oo} \mathbf{y}_{1:n} - \mathbf{y}_{1:n}^T A_{op} A_{pp}^{-1} A_{po} \mathbf{y}_{1:n} = \mathbf{y}_{1:n}^T (A_{oo} - A_{op} A_{pp}^{-1} A_{po}) \mathbf{y}_{1:n}$.

To match the standard parametrization of a multivariate $t$-distribution, we need to divide the quadratic form by this value so the minimum is normalized to 1. 
$$\propto \left(\frac{{\mathbf{y}^*_{1:m}}^T A_{pp} {\mathbf{y}^*_{1:m}} + 2 \mathbf{y}_{1:n}^T A_{op} {\mathbf{y}^*_{1:m}} + \mathbf{y}_{1:n}^T A_{oo} \mathbf{y}_{1:n}}{\mathbf{y}_{1:n}^T (A_{oo} - A_{op} A_{pp}^{-1} A_{po}) \mathbf{y}_{1:n}}\right)^{(d-k)/2}$$
Since the support is $\mathbb{R}^m$, this immediately gives us $k-d=\nu + m$. So $\nu = k - m - d = n - d$. Next, we get
$$\Sigma = \frac{\mathbf{y}_{1:n}^T (A_{oo} - A_{op} A_{pp}^{-1} A_{po}) \mathbf{y}_{1:n}}{n - d} A_{pp}^{-1}, \qquad \mu = -A_{pp}^{-1} A_{po} \mathbf{y}_{1:n}$$
\end{document}